\documentclass[12pt,reqno]{amsart}
\usepackage{graphicx} 
\usepackage{amsmath, amssymb,enumitem}
\usepackage[left=1.25in, right=1.25in, top=1.25in, bottom=1.25in]{geometry}
\usepackage{comment}
\usepackage{setspace}
\usepackage{mathrsfs}
\usepackage{amsthm}
\usepackage{appendix}
\usepackage{tikz}
\usepackage{enumitem}
\usepackage{subcaption}
\newcommand{\AM}{\operatorname{argmax}}

\newtheorem{theorem}{Theorem}
\newtheorem{prop}{Proposition}
\newtheorem{definition}{Definition}
\newtheorem{corollary}{Corollary}
\newtheorem{remark}{Remark}
\newtheorem{lem}{Lemma}
\newtheorem{example}{Example}

\title{Multi-to -one dimensional and semi-discrete screening}
\date{\today}
\author{Omar {Abdul Halim}} 
\email{oabdulh1@ualberta.ca, pass@ualberta.ca}
\author{Brendan Pass}
\thanks{The work of OAH was completed in partial fulfillment of the requirements for a doctoral degree in
mathematics at the University of Alberta. BP is pleased to acknowledge the support of Natural Sciences and
Engineering Research Council of Canada Discovery Grant numbers 04658-2018
and 04864-2024.}
\address{University of Alberta, Edmonton, Alberta, Canada}
\begin{document}

\begin{abstract}
We study the monopolist's screening problem with a multi-dimensional distribution of consumers and a one-dimensional space of goods.  We establish general conditions under which solutions satisfy a structural condition known as nestedness, which greatly simplifies their analysis and characterization.  Under these assumptions, we go on to develop a general method to solve the problem, either in closed form or with relatively simple numerical computations, and illustrate it with examples.  These results are established both when the monopolist has access to only a discrete subset of the one-dimensional space of products, as well as when the entire continuum is available.  In the former case, we also establish a uniqueness result.
\end{abstract}
\maketitle
\section{Introduction}

The monopolist's, or principal-agent, problem plays a crucial role in economic theory.  Following, for example, Wilson \cite{Wilson93},
the problem can be described as follows (although other interpretations are possible as well): a monopolist sells goods from a set $Y$ to a collection of consumers $X$.  Knowing the cost $c(y)$ to produce each good $y \in Y$, the preference $b(x,y)$ of each potential consumer $x \in X$   for each good $y \in Y$ and the relative frequency $f(x)$ of consumer types, her goal is to choose which goods to produce, and the prices to charge for them 
so as to maximize her profits.

This nonlinear pricing problem is well understood when both consumer types and goods have only one dimension of heterogeneity, at least under the celebrated Spence-Mirrlees condition on preferences \cite{mirrlees, MaskinRiley84, MUSSARosen78}.  In contrast, scenarios where consumers and/or goods exhibit multi-dimensional heterogeneity, known as multi-dimensional screening problems in the literature, are much more challenging, and despite considerable efforts and achievements by many authors, are still not well understood.  A general result of Carlier ensures existence of an optimal pricing strategy \cite{CARLIER2001}.  Among an expansive literature, we mention a seminal contribution of Rochet-Chone \cite{RochetChone98} introducing a complicated general approach to the problem when preferences are linear in types, and proposing a solution to an example where both consumer types and products are two-dimensional.  In the process, they discovered the \emph{bunching} phenomena in which different types choose the same good at optimality.  McCann-Zhang \cite{McCannZhang24} developed delicate duality and free boundary tools, leading to a refinement of the Rochet-Chone solution, while, in another direction, work of Figalli-Kim-McCann \cite{FigalliKimMcCann11} uncovered conditions under which the problem is a concave maximization for general preferences. Noldeke-Samuelson \cite{NoldekeSamuelson18} and McCann-Zhang \cite{McCannZhang19} have also extended existence and uniqueness results to problems where consumers' utilities are non-linear in prices. 
Much of this research exploits, either directly or indirectly, a connection to the mathematical problem of optimal transport (or, equivalently, the economic problem of matching under transferable utility) \cite{Santambrogio, Galichon16}.

We focus here on the case where consumer types are multi-dimensional (in fact, two-dimensional in the majority of the paper) but goods are one-dimensional.  Such models have already seen a fair bit of attention in the literature \cite{Deneckere, laffont, basov-2003}, likely because they are the simplest setting in which one can explore the effects of consumers' multi-dimensional heterogeneity, and, as highlighted by Basov \cite{basov-book}, it is natural to consider problems where types are higher dimensional than goods, reflecting the high degree of idiosyncrasy in consumers' tastes.

In the simpler setting of optimal transport (OT), the second named author, together with Chiappori and McCann recently developed a condition, known as \emph{nestedness}, under which the solution to the OT problem between a high dimensional source measure and a one-dimensional target can be characterized in a very simple way, and in fact be solved almost explicitly \cite{ChiapporiMcCannPass17}.  As solutions to the monopolist's problem indeed solve an optimal transport problem between the distribution $\mu=f(x)dx$ of agents and the distribution $\nu$ of purchased goods at optimality, with surplus given by their preference function $b(x,y)$, one might hope that nestedness is present in the monopolist's problem as well, and expect it to greatly simplify the analysis if so.   However,  checking nestedness is far from straightforward; in the standard OT problem, it depends on the interaction between the two measures $\mu$ and $\nu$ to be matched and the surplus function $b$. Since in the monopolist's problem $\nu$ is \emph{endogenous}, it is not possible to check nestedness directly.

The main contribution of this paper is to establish conditions under which solutions to the monopolist's problem are nested, and to exploit the resulting structure to analyze the solution.  We begin by working in a semi-discrete setting, where we only allow a finite number of goods, chosen from the original one-dimensional space of allocations (although, as mentioned below, several results will eventually be translated back to the continuous goods setting).  Though problems with a finite allocation space have certainly been considered before, they do not seem to have been explored in our multi-to one-dimensional setting.  As a side contribution, we develop a theory of optimal transport in this setting analogous to \cite{ChiapporiMcCannPass17}, including a condition (named discrete nestedness) under which the problem can be solved nearly explicitly.  This theory is, we believe, of independent interest.  Turning back to the monopolist's problem, the semi-discrete framework has significant technical advantages, as it makes perturbation arguments, commonplace in the mathematical calculus of variations, much simpler.  We believe that it also makes the economic interpretation of the nested structure more transparent; in the monopolist's setting, discrete nestedness essentially means that while a consumer may be indifferent between two goods, they will \emph{never} be indifferent among three or more.  Alternatively, discrete nestedness can be expressed as follows: when faced with an optimal pricing schedule, whenever an agent prefers the $i$th  good to the $i+1$-th one, (s)he will necessarily also prefer the $j$th to the $j+1$ -th one as well, for all $j \geq i$.  This is an easy consequence of the Spence-Mirrlees condition when types are one-dimensional, but does not hold in general in higher dimensions. 

We show that, under our conditions, solutions may often be found in closed form, and, when this is not possible, they can be found extremely easily numerically.  We also develop a uniqueness result; this is particularly notable, as  uniqueness of solutions in multi-dimensional monopolist problems is a fairly delicate issue. Indeed, strict concavity of the problem (a useful tool for establishing uniqueness, if present) requires very strong conditions on $b$; in fact, these conditions essentially cannot hold in the unequal dimensional setting we work in \cite{Pass2012}.  We also show by an approximation argument that a nested solution to the continuous problem also exists, and closed form solutions can sometimes be obtained by discrete approximations as well.

We pause now briefly to discuss the connection between our work and other multi-to one-dimensional screening research.  Laffont-Maskin-Rochet \cite{laffont} solved an example with a particular preference function and distribution of agents characteristics. 
As was highlighted by Rochet-Chone \cite{RochetChone98}, a key insight uncovered by their solution is that while bunching is necessary, it is possible to aggregate the codimension 1 sets of agents choosing the same product, and then the solution in the new, aggregated one-dimensional type space solves a classical one-dimensional problem (see also Section 3 in McAfee and McMillan \cite{McAfee}).  The difficulty is that the aggregation process is endogenous.  In fact, as shown by one of the present authors, the aggregation can be chosen canonically \emph{only} when the preference function $b$ has an \emph{index} form, $b(x,y) = \tilde b(I(x),y)$ where $I:X \rightarrow \mathbb{R}$, in which case the problem really does reduce to a one-dimensional one \cite{Pass2012}; in our nomenclature here, solutions are automatically nested when $b$ has an index form, as shown in \cite{ChiapporiMcCannPass17}.  
Somewhat similarly, Deneckere and Severinov \cite{Deneckere} demonstrate that solutions can be found by solving a certain one-dimensional optimal control problem, with an endogenous distribution of goods, 
and develop techniques which can solve certain examples fairly explicitly. Seen in this light, our work identifies general conditions under which the aggregation has a particular special form and can therefore be found in a tractable way\footnote{In particular, when a continuum of products is available, the aggregation is continuous for nested models.  Perhaps more striking is that in general, the aggregation may not respect the order of the aggregated type space, albeit for a negligible set of agents: an agent type $x \in X\subset \mathbb{R}^m$ may be matched to two aggregated types, $t_\pm \in \mathbb{R}$, but \emph{not} to those  $t$ in between, $t_- <t<t_+$.  This cannot happen for nested models. }. Consequently, when nestedness is present (as is the case under the conditions we identify) solutions can be easily found, either analytically or via simple numerics, without resorting to solving partial differential equations and free boundary problems as in \cite{RochetChone98}, or leaning on the complex calculations in \cite{Deneckere}.  

We also note that, in order to keep our arguments as manageable as possible, we work under various simplifying hypotheses; types are two-dimensional, and preferences are linear in types  -- see Section \ref{sect: nestedness in semi-discrete monopolist's}).  Even with the present assumptions, our proofs are fairly involved technically. However, the notion of discrete nestedness makes perfect sense more generally, and we believe our approach may prove useful in the future in other situations as well, provided the allocation space remains one-dimensional.

The manuscript is organized as follows.  In the next section, we provide a precise formulation of the monopolist's problem we will study and introduce a semi-discrete analogue of the notion of nestedness introduced in \cite{ChiapporiMcCannPass17} for the monopolist's problem.  In Section 3, we focus on the monopolist's problem with a two-dimensional set of consumers and a finite set of goods, chosen from a one-dimensional continuum.  We introduce the (somewhat technical) assumptions we will need, state our main result,  illustrate the role of our assumptions with a couple of examples, and then present as well as briefly discuss several intermediate results which are vital ingredients in the proof of nestedness (the proof itself is in an appendix), but, we believe, are also of interest in their own right.  In Section 4  we present an alternate, but closely related, characterization of solutions in the semi-discrete context, which allows us to find closed form solutions in certain cases, and compute solutions very efficiently in others.  
It also allows us to prove a uniqueness result, which we present in Appendix~\ref{uniqueness}. Section 5  extends  our main result to a continuum of products. The connection between the monopolist's problem and optimal transport, which underlies our proofs, is presented in Appendix~\ref{connection to ot}.

As mentioned above, proofs of results stated in the body of the paper are relegated to Appendix B.

\section{Formulation of the monopolist's problem}
Consider a monopolist who produces products $y \in Y\subset \mathbb{R}^n$ with $n$-dimensional qualities. She deals with an $m$-dimensional set of agents $X\subset \mathbb{R}^m$ whose relative frequency is given by an absolutely continuous (with respect to Lebesgue measure) probability measure $\mu(x)$ with density $f(x)$. Let $c(y)$ be the cost of production of product $y,$ and the function $b(x,y)$ represent the preference of agent $x$ for product $y.$ For every pricing function $v\,:\, Y\to\,[0,\infty)$ that the monopolist puts on the products ($v(y)$ is the price of product $y$) we assume that the agents will choose an optimal choice of product $y^*(x)$ that maximizes their utility $b(x,y)-v(y).$ Then we define $u(x)=\max_{y\in Y} \,b(x,y)-v(y)=b(x,y^*(x))-v(y^*(x))$ to be the payoff function of agent $x.$ Under the generalized Spence-Mirlees condition \cite{Santambrogio} on $b$ (that is, injectivity of $y \mapsto D_xb(x,y)$ for each fixed $x$),  it is well known that there is exactly one $y:=y^*(x)$ that maximizes $b(x,y)-v(y)$ for almost every $x,$ and that the function $y^*$ is uniquely determined from $u$, $\mu$ almost everywhere. The agents can also choose to opt out, meaning they choose to not purchase any product. This is captured by an opt-out good $y_0$ which the monopolist produces for $0$ cost, $c(y_0)=0$, and cannot charge for, so that the pricing function is required to assign $v(y_0)=0.$ This implies that $u(x)\ge b(x,y_0)-v(y_0)=b(x,y_0).$ Now, for each agent of type $x,$ the monopolist's profit from this buyer is $v(y^*(x))-c(y^*(x))=b(x,y^*(x))-u(x)-c(y^*(x))$ and her total profits can be written as

\[\mathcal{P}(u)=\int_X (b(x,y^*(x))-u(x)-c(y^*(x)))f(x)dx.\] Hence, we define the monopolist's problem as follows,
\begin{equation}\label{PA}
\max_{u\in\mathcal{U},\, u\ge b(\cdot,y_0)}\mathcal{P}(u)
\end{equation}
where $\mathcal{U}:=\{u: X \rightarrow \mathbb{R}\,: \, u(x)=\max_{y\in Y} \,b(x,y)-v(y),\,\text{ for some function } v \}.$ 
\noindent 
For any $u \in \mathcal{U}$, one can associate a pricing function $v$ defined by $v(y)=\max_{x\in X}\{ b(x,y)-u(x)\},$ which yields the same profit in \eqref{PA}.

An important special case occurs when $m=n$ and $b(x,y)=x\cdot y.$ In this case, the optimal choice of product for agent $x$ is  $y^*(x)=Du(x)$, the gradient of $u$, %
which implies the condition that $Du(x)\in Y$  for all $x\in X.$ If the opt-out option is given by $y_0=0$, we can rewrite the monopolist's problem as follows
\[\max_{u\in\mathcal{U},\, u\ge0}\int_X (x\cdot Du(x)-u(x)-c(Du(x)))f(x)dx,\]
and the set $\mathcal{U}$ becomes the set of convex functions defined on $X,$ such that $Du\in Y.$

Another special case is when both $X$ and $Y$ are one-dimensional and $b$ satisfies the Spence-Mirrlees condition $\frac{\partial^2b}{\partial x \partial y} >0$; here, it is well known that, for any price schedule $v(y)$, the consumers' choice function $x \mapsto y^*(x) \in \AM_y [b(x,y)-v(y)]$ is monotone increasing~\cite{basov-book}.  This property can be expressed in various ways.  When $Y=\{y_0,y_1,...,y_N\}$ is discrete, one  way is that while a consumer $x$ may be indifferent between two \emph{adjacent} goods, $y_i, y_{i+1} \in \AM_y [b(x,y)-v(y)]$, they will never be indifferent between \emph{non-adjacent} goods; ie, if $|i-j| >1$, we cannot have $y_i, y_{j} \in \AM_y [b(x,y)-v(y)]$ (unless $y_k \notin \AM_y [b(x,y)-v(y)]$ for all $x$ and any $i<k<j$, in which case these $y_k$ can be neglected).  

%

Our interest here is largely in understanding how and when monotonicity generalizes in an appropriate sense to higher dimensional $X$ (with $Y$ still one-dimensional).   We will be interested in the case where $X\subseteq \mathbb{R}^2$, $Y \subseteq \mathbb{R}$ parametrizes a curve, $z(y)=(y,F(y))$, or a finite set of points along a curve, in $\mathbb{R}^2$, and $b(x,y)=x\cdot z(y).$

The notion of \emph{nestedness}, which can in some sense be understood as such a generalization, was introduced by the second named author, together with Chiappori and McCann \cite{ChiapporiMcCannPass17}, for optimal transport (or, equivalently, matching with transferable utility) problems between continuous measures on $X \subset \mathbb{R}^m$ and $Y \subset \mathbb{R}$.  
We will adapt this notion to the monopolist's problem and also develop a new formulation of nestedness which applies when the target space $Y$ is discrete.

When $Y=Y_N$ is discrete, we define the discrete level and sub-level sets as follows:
\begin{eqnarray}\label{eqn: discrete level sets} X_{=}^N(y_i,k):=&\{x\in X\,:\, b(x,y_{i+1})-b(x,y_{i})=k\},\\
X_{\le}^N(y_i,k):=&\{x\in X\,:\, b(x,y_{i+1})-b(x,y_{i})\le k\},\nonumber
\end{eqnarray}
and set $X_{<}^N(y_i,k):=X_{\le}^N(y_i,k)\setminus X_{=}^N(y_i,k).$ 
\begin{definition}
    We say that $u\in\mathcal{U}$ is  discretely nested  if 
    \[ X_{\le}^N(y_{i},v_{i+1}-v_i)\subseteq X_{<}^N(y_{j},v_{j+1}-v_j),\] 
     for all $i<j$   where $v_r=v(y_r)=\max_{x\in X}\{b(x,y_r)-u(x)\}$.
\end{definition}

\noindent 
Thus, the discrete nestedness condition ensures a consistent ordering of preferences: 
agents who prefer $y_i$ to $y_{i+1}$ (meaning $b(x,y_i)-v_i\ge b(x,y_{i+1})-v_{i+1}$) must also prefer each subsequent product $y_j$ 
to its successor $y_{j+1}$ for all indices $j>i$. 
 For a general pricing plan $v$, the set $X_i$ of agents choosing good $y_i$\footnote{More precisely, $X_i$ is the set of agents \emph{potentially} choosing $y_i$, since if equality $b(x,y_i) -v_i = b(x,y_j) -v_j$ holds, agent $x$ is indifferent between goods $y_i$ and $y_j$, and may choose either.  Under suitable assumptions on $b$ and $\mu$ the set of agents who are indifferent will be $\mu$ negligible.} is
$$
X_i = \{x \in X: b(x,y_i) -v_i \geq b(x,y_j) -v_j \text{ for all }j=0,1,...N\}.
$$
The structure of the sets $X_i$ and how they fit together may in general be very complicated (see Figure~\ref{fig:nonnested} for an example of what these regions could look like), as to determine $X_i$ one must compare $b(x,y_i) -v_i$ to each of the $N$ other  $b(x,y_j) -v_j$.  On the other hand, \emph{if} nestedness holds, we have 
\begin{eqnarray*}
    X_i &= &X_\leq^N(y_i,v_{i+1}-v_i) \setminus X_<^N(y_{i-1},v_{i}-v_{i-1})\\
    &=&\{x \in X: b(x,y_i) -v_i \geq b(x,y_j) -v_j \text{ for }j=i-1,i+1\}
\end{eqnarray*}
which can be identified by comparing $b(x,y_i) -v_i$ \emph{only} to $b(x,y_{i-1}) -v_{i-1}$ and $b(x,y_{i+1}) -v_{i+1}$~(see Figure~\ref{fig:nested}).



\begin{figure}[htbp]
  \centering
  \begin{subfigure}{0.48\linewidth}
    \centering
    \includegraphics[width=\linewidth]{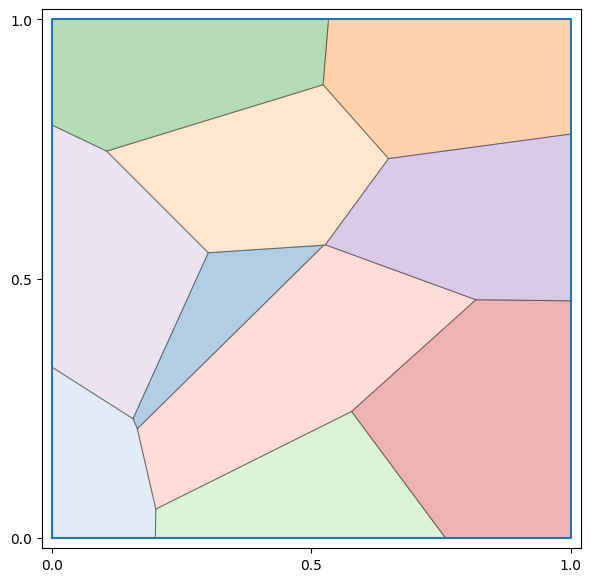}
    \caption{Illustration of the regions $X_i$ for a non-nested $u\in\mathcal{U}$.}
    \label{fig:nonnested}
  \end{subfigure}\hfill
  \begin{subfigure}{0.48\linewidth}
    \centering
    \includegraphics[width=\linewidth]{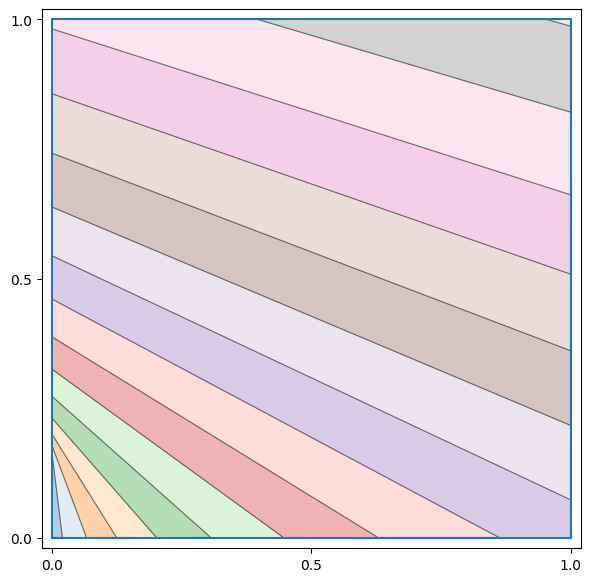}
    \caption{Illustration of the regions $X_i$ for a nested $u\in\mathcal{U}$.}
    \label{fig:nested}
  \end{subfigure}
  \caption{Comparison of regions $X_i$ for non-nested and nested $u\in\mathcal{U}$.}
  \label{fig:nested-comparison}
\end{figure}

\section{Nestedness of solutions the semi-discrete monopolist's problem}\label{sect: nestedness in semi-discrete monopolist's}
We are now ready to turn our attention to the structure of solutions to the monopolist's problem when the set $Y$ of available goods is finite with a one-dimensional structure.  We work in the particular setting described below.

Let $X=(0,1)^2$ and let $Y=[0,\Tilde{y}]$ and $z:[0,\Tilde{y}]\mapsto \mathbb{R}^2$ be the parametrization $z(y)=(y,F(y))$ for some $\Tilde{y}>0,$ where $F$ is an increasing convex function, and $F(0)=0.$ Let $Y_N=\{ y_i:\, 0\le y_i<y_{i+1}\le \Tilde{y} \text{ for  } 0\le i \le N\},$ where $y_0=0$, be a finite subset of $Y$, set $b(x,y)=x\cdot z(y),$ and let $y \mapsto c(z(y))$ be an increasing, convex cost function in $y$ such that $c(0)=0$. Let $\mu$ be a probability measure on $X$ with density function $f$ such that $\alpha\le f\le \|f\|_\infty$ for some $\alpha>0,$ with a bounded gradient, $Df:=(f_{x_1},f_{x_2}) \in L^{\infty}([0,1]^2)$. 


\begin{example}
 To illustrate the model, consider a monopolist manufacturing wool hats, differing across two qualities: their warmth, $z_1$ and durability $z_2$. These two qualities are modeled independently: a consumer might strongly prefer a very warm hat but care little about how long it lasts, while another consumer might prefer a moderately warm hat but place high importance on durability.
Consumers are represented by types $x=(x_1,x_2)\in X$, where $x_1$ measures how much a consumer values increased warmth, and $x_2$ measures the importance placed on durability; their preference function is then $b(x,z) = x_1z_1+x_2z_2.$
Now, it is reasonable to assume that both warmth and durability are actually determined  by the \emph{quality} $y$ of the wool used to manufacture the hats, as increasing functions $z_1(y),z_2(y)$.  Reparametrizing so that the warmth $z_1(y) =y$, and setting $z_2(y)=F(y)$ then leads to the preference function $b(x,y) =x \cdot (y,F(y))$.  If the manufacturer only has access to several fixed grades of wool, $y_1,...,y_N$, we recover a model of the form described above. 
\end{example}

Consider the following  technical assumptions on $c$, $\mu$ and $F:$ 
For all $z_k=z(y_k),\,\,z_i=z(y_i),\,z_j=z(y_j)$ such that $k<  i< j,$ 
\begin{enumerate}[label={(H\arabic*)}]
\item \label{condition2} $ \frac{c(z_j)-c(z_i)}{F(y_j)-F(y_i)}>\frac{c(z_i)-c(z_k)}{F(y_i)-F(y_k)},$\vspace{3pt}
  \item \label{condition3} $\frac{\frac{c(z_{i+1})-c(z_i)}{y_{i+1}-y_i}-\frac{c(z_i)-c(z_
  {i-1})}{y_i-y_{i-1}}}{\frac{F(y_{i+1})-F(y_i)}{y_{i+1}-y_i}-\frac{F(y_i)-F(y_{i-1})}{y_i-y_{i-1}}}>\frac{\frac{3}{2}\|f\|_\infty+\frac{\|f_{x_1}\|_\infty}{2}\Big(1+\frac{F(y_{i+1})-F(y_i)}{y_{i+1}-y_i} +\frac{c(z_{i+1})-c(z_i)}{y_{i+1}-y_i}\Big)}{\alpha},$
  
   \item \label{condition4} 
   \[\begin{array}{ll}
   \frac{F(y_{i+1})-F(y_{i})}{y_{i+1}-y_{i}}<&  \frac{2\alpha}{\|f\|_\infty\Big(2+\frac{\frac{F(y_{i+1})-F(y_{i})}{y_{i+1}-y_{i}}}{\frac{F(y_i)-F(y_{i-1})}{y_i-y_{i-1}}}\Big)+\|f_{x_2}\|_\infty\Big(1+2\frac{y_{i}-y_{i-1}}{F(y_{i})-F(y_{i-1})}+\frac{c(z_{i+1})-c(z_i)}{F(y_{i+1})-F(y_i)}\Big)}\times\\
   &\Bigg(1+\frac{\frac{c(z_{i+1})-c(z_i)}{F(y_{i+1})-F(y_i)}-\frac{c(z_{i})-c(z_{i-1})}{F(y_i)-F(y_{i-1})}}{\frac{y_{i}-y_{i-1}}{F(y_{i})-F(y_{i-1})}-\frac{y_{i+1}-y_i}{F(y_{i+1})-F(y_{i})}}\Bigg).
   \end{array}\]

\end{enumerate}
The first assumption expresses that $c$ is more convex than $F$, while the second locally quantifies the difference in convexity between $c$ and $F$ in terms of various quantities of interest in the problem. The third is a local bound on the derivative of $F$.
\begin{example}\label{ex1}
 If $c(z)=\frac{|z|^2}{2},$ $|z_{i+1}-z_i|=\frac{1}{N}$ and  $F(y)=Ay^2$ such that $A<\frac{1}{3},$  then $c$ and $F$ satisfy \ref{condition2}--\ref{condition4} when $f=1$ and $N>3.$  The calculation proving this is provided in an appendix.

\end{example}


Our main result on the semi-discrete monopolist's problem is the following.
 \begin{theorem}\label{disc}
    Under the hypotheses \ref{condition2}--\ref{condition4}, any solution $u\in \mathcal{U}$ of the monopolist's problem with data $(\mu, Y_N,c)$ is discretely nested. 
 \end{theorem}  
While conditions \ref{condition2} -- \ref{condition4} appear complicated, we stress that \emph{some} hypotheses are necessary to ensure nestedness of the solution, as the following example confirms. 
\begin{example}\label{ex2}
    In the case of a uniform measure $\mu$, \emph{if} the solution is discretely  nested, it can in fact be determined explicitly; this is shown in Section~\ref{subsect: explicit solution}.

 Consider now  Example~\ref{ex1} with $A = \frac{1}{2.9}$; it is not hard to show that  \ref{condition3} \emph{fails} for large enough $N$, and we claim that, in fact, nestedness of the solution fails as well. In Section \ref{subsect: explicit solution}, we attempt to evaluate the solution under the assumption of discrete nestedness, using the explicit solution mentioned above.  
The resulting structure is \emph{not} discretely nested, showing that the explicit construction fails to yield the solution to the problem and implying that the solution itself must violate discrete nestedness (see equation \eqref{eqn: non-nested example} below).

\end{example}

\subsection{Significant intermediate results}  
The proof of Theorem \ref{disc} is fairly long.  It is divided into several intermediate results and lemmas.  All proofs are relegated to the appendix; however, we state here several of the intermediate results which we believe are of independent interest, and briefly explain their significance.

  
   The first of these is Theorem \ref{between}, which expresses that, for the optimal pricing plan, the set of goods which are actually produced and purchased by some consumer is consecutive.
   In what follows, for $u\in\mathcal{U}$ we  note that, up to negligible sets, $X_i=\{x\in X: Du(x)=z_i=z(y_i)\}$ for  $0\le i\le N$ (recall that $X_i$ corresponds to the consumers buying product $y_i$).

   \begin{theorem}[Purchased goods are consecutive]\label{between}
     Assume that $c$ and $F$ satisfy  \ref{condition2} and let $u$ be a solution of the monopolist's problem with data $(\mu, Y_N,c).$ Then, if $y_k$ and $y_j$ are 2 products such that $\mu( X_k)$ and $\mu( X_j)$ are positive, then $\mu( X_p)$ is positive for all $k\le p\le j.$ 
\end{theorem}
 The proof of this result itself will require several lemmas, most of which are developed in the appendix.  We do present one of them here, stating that all purchased goods are purchased by consumers on the bottom or right hand side of the boundary. 
\begin{lem}\label{partition}
    Let $u\in \mathcal{U},$ and let $(y_{i_k})$  be the products with $\mu(X_{i_k})>0$  such that $0\le i_k<i_{k+1}\le N$ for all $k.$ Then,
    \begin{enumerate}
      \item $\overline{X_{i_k}}\cap \Big(([0,1]\times\{0\}) \cup(\{1\}\times[0,1])\Big)\neq\emptyset.$
     \item $\overline{X_{i_{k+1}}}\cap\overline{X_{i_k}}\cap \Big(([0,1]\times\{0\}) \cup(\{1\}\times[0,1])\Big) \neq \emptyset.$
  
    \end{enumerate}
\end{lem}

The key property behind nestedness is that the \emph{indifference curves} $\overline X_i \cap \overline X_j=\{x\in \overline{X}: u(x)=b(x,y_i)-v_i=b(x,y_j)-v_j\}$ cannot intersect each other, which the following result asserts.

\begin{theorem}[Indifference curves cannot intersect at optimality]\label{thm: no intersection}
    Under conditions \ref{condition2}--\ref{condition4}, no two indifference curves arising from an optimal $u$ can intersect within $\overline{X}$.
\end{theorem}

\section{An alternate characterization of discretely nested solutions}
In this section, we offer an alternate characterization of solutions of the semi-discrete problem, assuming discrete nestedness, in terms of the points where the indifference curves intersect the upper part of the boundary.  This formulation has several advantages.  First, it allows us to establish uniqueness of the solution, under additional hypotheses (Appendix~\ref{uniqueness}).  Second, the first order conditions in these new variables are quite simple, and in some cases (such as when $\mu$ is uniform) can even be solved explicitly.  Finally, even when explicit solutions are not possible, solving these first order conditions yields a simple and efficient numerical method, which we develop and illustrate in the following section.

In what follows, assume that $c(z_1)>F(y_1),$ in addition to the assumptions laid out in Section \ref{sect: nestedness in semi-discrete monopolist's}. From the discrete nestedness of the solution, the monopolist's problem is equivalent to maximizing the profit function $\mathcal{P}$ over convex  utility functions $u\ge0$ such that $u$ is discretely nested.
In this case, each indifference segment intersects either $\{0\}\times[0,1]$ or $[0,1]\times\{1\}.$\footnote{In fact, the intersection points are all in $[0,1]\times\{1\},$ as Lemma \ref{upper_segment} below asserts.  However, the proof of this fact actually relies on the formulation below allowing for intersection points with $\{0\}\times[0,1]$ as well, so it is necessary to develop this formulation as well.} We parametrize $\{0\}\times[0,1]\cup[0,1]\times\{1\}$ by $x:[0,2]\to\{0\}\times[0,1]\cup[0,1]\times\{1\}$ where $x(t)=(0,t)$ if $t\in[0,1]$ and $x(t)=(t-1,1)$ if $t\in[1,2].$ From this we can write the prices in terms of the points $x(t_i)$ where the indifference curves intersect this portion of the boundary. The  indifference segment between $ X_i$ and $ X_{i+1},$ satisfies
\[x(t_i)\cdot z_i-v_i=x(t_i)\cdot z_{i+1}-v_{i+1}\]
which implies that $v_{i+1}=x(t_i)\cdot(z_{i+1}-z_i)+v_i$ and by induction and the fact that $v_0=0,$ we get 
\[v_i=\sum_{k=0}^{i-1}(x(t_{k})\cdot (z_{k+1}-z_{k})).\]

Now we can write the profit function as
\[\mathcal{P}(t_0,\dots,t_{N-1})=\sum_{i=1}^{N-1}(v_i-c(z_i))\mu( X_i)=\sum_{i=1}^{N-1}(\sum_{k=0}^{i-1}(x(t_{k})\cdot (z_{k+1}-z_{k}))-c(z_i))\mu( X_i).\]
\begin{lem}\label{upper_segment}
The upper  intersection points    $(x(\overline{t_i}))$ between the indifference segments of the solution and $\partial X$  are all in $[0,1]\times\{1\}.$ 
\end{lem}

Due to this lemma, we can redefine the parametrization $\overline{x}:\,[0,1]\to [0,1]\times\{1\}$ where $\overline{x}(t)=(t,1)$ and so the profit function becomes 
\begin{equation}\label{eqn: definition of P(t)}
    \mathcal{P}(t_0,\dots,t_{N-1})=\sum_{i=1}^{N-1}(v_i-c(z_i))\mu( X_i)=\sum_{i=1}^{N-1}(\sum_{k=0}^{i-1}(\overline{x}(t_{k})\cdot (z_{k+1}-z_{k}))-c(z_i))\mu( X_i).
\end{equation}

The next lemma characterizes those goods which are produced at optimality as exactly those goods $y_i$ which the highest end consumer $x=(1,1)$ prefers to the next highest good $y_{i-1}$ when both are offered at cost.
\begin{lem}\label{m} 
Let $u$ be a solution of \eqref{PA}.  Then $\mu(X_0)>0$ and, for $i \geq 1$, $\mu(X_i) >0$ if and only if 
\begin{equation}\label{eqn: preference at cost}
b((1,1),z_{i}) -b((1,1),z_{i-1}) -c(z_{i})+c(z_{i-1}) >0.
\end{equation}

\end{lem}
The first assertion of Lemma \ref{m} is a manifestation of the well known \emph{principle of exclusion} in multi-dimensional screening, identified by Armstrong \cite{Armstrong96}, although the proof in the current semi-discrete setting is much simpler.
Inequality \eqref{eqn: preference at cost} is equivalent to
$$
1+\frac{F(y_{i}) -F(y_{i-1}) }{y_i-y_{i-1}}-\frac{c(y_{i}) -c(y_{i-1}) }{y_i-y_{i-1}}>0.
$$
Since the left hand side is decreasing in $i$, the set of $i$ that satisfies it is consecutive, starting at $i=1$ and ending at some $M \leq N$, where $M$ is the largest index satisfying \eqref{eqn: preference at cost}.

Let $\mathcal{B}=\{(t_0,\dots,t_{M-1})\in (0,1)^M:\, 0\le t_i<t_{i+1} \text{ for all } 0\le i<M-1\}$ and we define $\mathcal{P}:\overline{\mathcal{B}}\mapsto [0,\infty)$ by \eqref{eqn: definition of P(t)}.

    Note that Theorem \ref{disc} and \ref{thm: no intersection} together with Lemmas \ref{upper_segment} and \ref{m},  imply the following 
\begin{theorem}\label{p(t)}
There exists $(\overline{t}_i)\in \mathcal{B}$ that maximizes $\mathcal{P}(t_0,\dots , t_{M-1})$, 
and the corresponding profit satisfies
\begin{equation}\label{Eqn: reformuation problem}
\mathcal{P}(\overline{t}_0,\dots, \overline{t}_{M-1})
=\max_{u\in \mathcal{U},\, u\ge 0}\mathcal{P}(u).
\end{equation}
\end{theorem}

\begin{remark}
    Theorem \ref{disc} implies that the maximizer $(\overline{t}_i)$ of $\mathcal{P}$ is in $\overline{\mathcal{B}}$ which means  $\overline{t}_i\le \overline{t}_{i+1}$ for all $0\le i<M-1.$ Theorem \ref{thm: no intersection} implies that $t_i<t_{i+1}$ and so $(\overline{t}_i)\in \mathcal{B}.$
\end{remark}

\subsection{Explicit solutions}\label{subsect: explicit solution}
 Note that the nested solution $u$ of the monopolist's problem is defined by
    \[u(x)=x\cdot z_i-\sum_{k=0}^{i-1}((\overline{t}_{k},1)\cdot (z_{k+1}-z_{k})),\]
    when $x\in X_i$ where
\vspace{5pt}

\resizebox{0.97\textwidth}{!}{$
\hspace{-15pt}X_i=\begin{cases}
    \Big\{(x_1,x_2)\in X:\,x_2< -\frac{y_{i+1}-y_{i}}{F(y_{i+1})-F(y_{i})}(x_1-\overline{t}_{i})+1\Big\}& \text{if } i=0,\\
    \Big\{(x_1,x_2)\in X:\, -\frac{y_i-y_{i-1}}{F(y_i)-F(y_{i-1})}(x_1-\overline{t}_{i-1})+1<x_2<-\frac{y_{i+1}-y_{i}}{F(y_{i+1})-F(y_{i})}(x_1-\overline{t}_{i})+1\Big\}&\text{if } 0<i<M-1,\\
    \Big\{(x_1,x_2)\in X:\, x_2>-\frac{y_{M}-y_{M-1}}{F(y_{M})-F(y_{M-1})}(x_1-\overline{t}_{M-1})+1\Big\}& \text{if } i=M-1,
\end{cases}
$}
\vspace{5pt}    

and $(\overline{t}_i)_{i=0}^{M-1}$ is a maximizer of the function $\mathcal{P}(t_0,\dots,t_{M-1}).$

Theorem \ref{p(t)} implies that solving the monopolist's problem under the assumptions in this paper boils down to finding the root $\bar t_i$ of the derivative of $\mathcal{P}$ with respect to each $t_i$. It is straightforward to see that these equations decouple from each other; that is, each $\frac{\partial \mathcal{P}}{\partial t_i}$ depends \emph{only} on $t_i$ and not on $t_j,j\neq i$ (the explicit calculation is done in Appendix~\ref{uniqueness}). 
These equations can therefore each be solved independently, making the problem considerably more tractable.  In certain cases it can be solved in closed form.  Indeed, if $f(x) =1$, so that $\mu$ is the uniform measure, we get:

\begin{equation}\label{eqn: explicit t values}
t_i^N=\frac{1}{2}-\frac{3}{4}\frac{F(y_{i+1})-F(y_i)}{y_{i+1}-y_i}+\frac{1}{2}\frac{c(z_{i+1})-c(z_i)}{y_{i+1}-y_i},\end{equation}
 if $t_i^N+\frac{F(y_{i+1})-F(y_i)}{y_{i+1}-y_i}<1,$ otherwise,
 \begin{equation}\label{eqn: explicit t values 2}
 t_i^N=\frac{1}{3}-\frac{2}{3}\frac{F(y_{i+1})-F(y_i)-(c(z_{i+1})-c(z_i))}{y_{i+1}-y_i}.\end{equation}

These solutions for various choices of $N$, F and $c$ are illustrated in Figures \ref{unif} \ref{unif2} and \ref{nonnested}.  Note that the conditions ensuring nestedness described in Example \ref{ex1} are \emph{satisfied} in Figures \ref{unif} and \ref{unif2}, but \emph{fail} in Figure \ref{nonnested}.  Therefore, Figures \ref{unif} and \ref{unif2} depict exact solutions to the monopolist's problem.  On the other hand, \emph{if} the solution was nested for the $N$, $F$ and $c$ in Figure \ref{nonnested}, Theorem \ref{p(t)} would imply that the solution be given by \eqref{eqn: explicit t values} and \eqref{eqn: explicit t values 2}.  However, these choices of $t$ do not result in a nested structure (note the intersecting level curves in Figure \ref{nonnested}).  Therefore, the solution for the choices of $N$, F and $c$ in Figure \ref{nonnested} cannot be nested.  Similar reasoning applies to Example \ref{ex2}, taking $N=20$ and $A=\frac{1}{2.9}$. Using the computed values  
\begin{equation} \label{eqn: non-nested example}
(t_i^N) = (0.48582, 0.48525, 0.48502, 0.48522, 0.48591, \dots),
\end{equation}
we observe that the sequence is not monotonic: specifically, $ t_0^N > t_1^N $ and $ t_1^N > t_2^N $. This violates the monotonicity guaranteed in the nested case, and we therefore conclude that the solution in Example \ref{ex2} is not nested either.

\subsection{Numerical computation}\label{num}
Several numerical algorithms for screening problems have been developed in the literature, under different assumptions \cite{EkelandMorenoBromber10, Mirebeau16,CarlierDupuiRochetThanassoulis24}.  However, nestedness and the reformulation \eqref{Eqn: reformuation problem} leads to a  simpler computational scheme.
 
 Even when the roots of $\frac{\partial \mathcal P}{\partial t_i}$ 
 cannot be found by hand, the fact that the equations decouple (that is, $\frac{\partial \mathcal P}{\partial t_i}$ does not depend on $t_j$ for $j \neq i$) means they can be easily found  numerically; Theorem \ref{p(t)}  ensures that these roots correspond to the solution of the monopolist's problem.  We illustrate this by solving an example in Figure \ref{gaus}.  In general, we expect this approach, which amounts to solving $N$ independent one-dimensional equations, to be far more efficient than other methods whenever it is applicable. 

\begin{figure}[!h]
  \centering
  \begin{subfigure}[t]{0.48\textwidth}
    \centering
    \includegraphics[height=5cm]{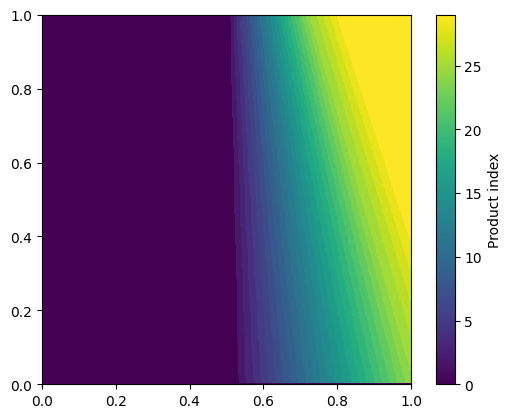}
    \caption{Regions $X_i$ for $N = 28$, $F(y) = \frac{y^2}{6}$, $c(z) = \frac{|z|^2}{2}$, $\mu$ uniform.}
    \label{unif}
  \end{subfigure}
  \hfill
  \begin{subfigure}[t]{0.48\textwidth}
    \centering
    \includegraphics[height=5cm]{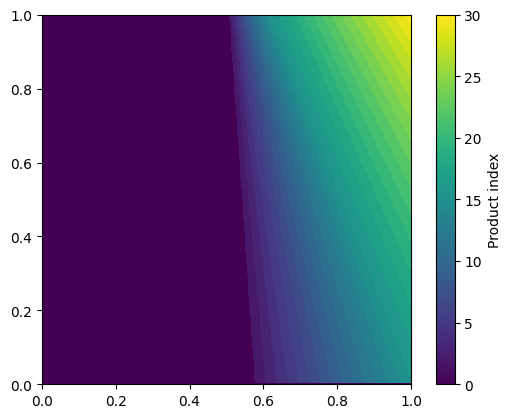}
    \caption{Regions $X_i$ for $N = 30$, $F(y) = \frac{y^2}{4}$, $|z_{i+1} - z_i| = \frac{1.4}{30}$, $c(z) = \frac{|z|^2}{2}$, $\mu$ uniform.}
    \label{unif2}
  \end{subfigure}

  \vspace{0.5cm}

  \begin{subfigure}[t]{0.48\textwidth}
    \centering
    \includegraphics[height=5cm]{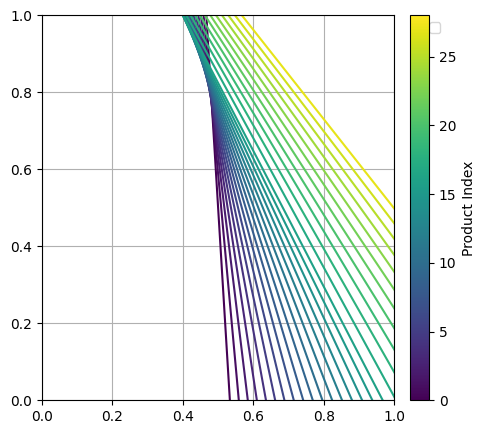}
    \caption{The level curves $X^N_=(y_i, v_{i+1}-v_i)$ with $N = 28$, $F(y) = \frac{y^2}{2}$, $c(z) = \frac{|z|^2}{2}$, $\mu$ uniform.}
    \label{nonnested}
  \end{subfigure}
  \hfill
  \begin{subfigure}[t]{0.48\textwidth}
    \centering
    \includegraphics[height=5cm]{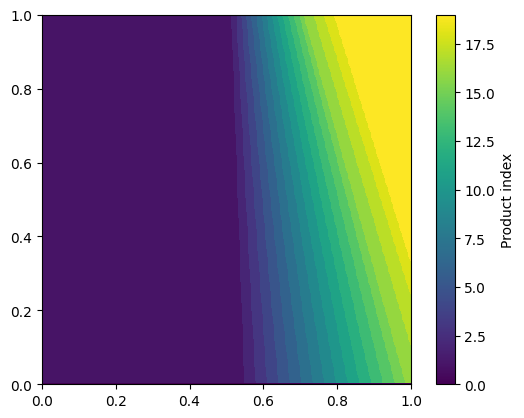}
    \caption{Regions $X_i$ for $N = 18$, $F(y) = \frac{y^2}{6}$, $c(z) = \frac{|z|^2}{2}$, $\mu$ Gaussian (normalized).}
    \label{gaus}
  \end{subfigure}

  \caption{Comparison of regions $X_i$ and indifference curves behavior under varying model parameters.}
\end{figure}

\section{Nestedness for a continuum of products}
We turn now to the case where the set of available products is continuous. When $Y\subset\mathbb{R}$ parametrizes a curve, we define nestedness of the monopolist's problem as follows:

The continuous analogues of the discrete level and sublevel sets \eqref{eqn: discrete level sets} are defined in terms of marginal preference functions:

\[X_{=}(y,k):=\Big\{x\in X\,:\, \frac{\partial b}{\partial y}(x,y)=k\Big\},
X_{\le}(y,k):=\Big\{x\in X\,:\, \frac{\partial b}{\partial y}(x,y)\le k\Big\},\]
and $X_{<}(y,k):=X_{\le}(y,k)\setminus X_{=}(y,k).$ 

 The continuous analogue of the discrete finite differences $v_{i+1}-v_i$ in price are the marginal differences $v'(y)$.  In general, however, note that pricing functions of the form $v(y)=\max_{x\in X}\{b(x,y)-u(x)\}$ may not be everywhere differentiable, but they are semi-convex, meaning that they have \textit{subdifferentials} everywhere (see, for instance,~\cite{Rockafellar1970}), and are in fact differentiable Lebesgue almost everywhere. 
We define $v'_+(y)$ and $v'_-(y)$ such that the subdifferential of $v$ at $y$ is  $\partial v(y)=[v'_-(y),v'_+(y)]$ where $v(y)=\max_{x\in X}\{b(x,y)-u(x)\}$ for some $u\in\mathcal{U}.$  Note that wherever $v$ is differentiable, $\partial v(y) =\{v'(y)\}$.

\begin{definition}
    We say that $u\in\mathcal{U}$ is nested  if 
    \[ X_{\le}(y,v'_+(y))\subseteq X_{<}(y',v'_-(y')),\] 
     for all $y<y'$   where  $v(s)=\max_{x\in X}\{b(x,s)-u(x)\}$.
\end{definition}
 \begin{remark}
The notion of nestedness in continuous optimal transport problems yields a very simple characterization of solutions, allowing problems to be solved in essentially closed form \cite{ChiapporiMcCannPass17} (this is reviewed in detail in Appendix \ref{connection to ot}, toegether with an analogous characterization in the discrete case).

In the context of the monopolist's problem with a continuum of goods considered here, in analogy with the discrete case, nestedness makes construction of the solution from the pricing function $v$ very simple.  Generally speaking, the envelope theorem implies that the set $X_y:=\{x: b(x,y) -v(y) \geq b(x,y')-v(y') \forall y' \in Y\}$ of consumers choosing good $y$ satisfies 
    $$
    X_y \subseteq \cup_{k \in \partial v(y)}X_=(y,k).
    $$
\end{remark}
For general pricing functions, this inclusion may be strict, making the reconstruction of each $X_y$ from $v$ complicated (as comparisons between $b(x,y) -v(y)$ and  with every other $b(x,y')-v(y')$ are required) .  However, the nestedness criterion, if present, ensures that each $x$ can be in $X_=(y,k) $ with $k \in \partial v(y)$ for \textit{only one} $y$; this implies the \textit{equality} $X_y = \cup_{k \in \partial v(y)}X_=(y,k) $.  It is therefore much simpler to find $X_y$ from $v$ (as construction the sets $X_=(y,k)$ for $k \in \partial v(y)$ requires knowledge of the behaviour of $v$ only at $y$). This, in turn, allows for a much simpler characterization of solutions to \eqref{PA}; see, for instance, the discussion below Corollary \ref{c1}.

The setting is  exactly as laid out in Section \ref{sect: nestedness in semi-discrete monopolist's}, except that the entire set $Y=[0,\Tilde{y}]$ of goods is available.

We will approximate the set $Y$ with the discrete set $Y_N$, where the points $z_i = (y_i,F(y_i))$ are equally spaced, so that the arclength of $\{(y,F(y)):\, y_i\le y\le y_{i+1}\}$ is $\frac{L}{N}$, where $L$ is the arclength of the curve $\{(y,F(y)):\, 0\le y\le\Tilde{y}\}$.
This approximation will be used to prove the main result of this section, which is as follows.
\begin{theorem}\label{cont.nestedness}
    Assume that for all $N$ large enough $f,$ $F$ and $c$ satisfy  \ref{condition2}--\ref{condition4}. Then there exists a nested solution to the monopolist's problem with data $(\mu, Y,c)$.
 \end{theorem}  
We note a consequence on the regularity of the consumers' utility function $u.$
\begin{corollary}\label{c1}
    Under the assumptions in Theorem \ref{cont.nestedness}, there exists a continuously differentiable solution $u\in C^1(X)$ of the monopolist's problem with data $(\mu, Y,c)$.
\end{corollary}


When $\mu$ is uniform, from \eqref{eqn: explicit t values} and \eqref{eqn: explicit t values 2}, we can deduce the solution of the continuous problem which is the limit of the solutions $u_N$ as $N\to\infty.$ For $y\in Y$, there exists a sequence $(y_{i_N})$ converging to $y$ such that $y_{i_N}\in Y_N$ for all $N.$ Then, the sequence $(t_{i_N}^N)$ defined in \eqref{eqn: explicit t values} and \eqref{eqn: explicit t values 2} converges to $t_y$ so that 
\[t_y=\frac{1}{2}-\frac{3}{4}F'(y)+\frac{1}{2}(c_{x_1}(y,F(y))+F'(y)c_{x_2}(y,F(y)))\]
if $t_y+F'(y)<1,$ otherwise
\[t_y=\frac{1}{3}-\frac{2}{3}(F'(y)-(c_{x_1}(y,F(y))+F'(y)c_{x_2}(y,F(y)))).\]
Hence, the optimal map for the continuous problem matches all $x\in L(y):=\{x\in X:\, x_2=\frac{1}{F'(y)}(x_1-t_y)+1\}$ to the point $y\in Y.$

\section{Conclusion}
This paper introduces  analogues of the nestedness criterion introduced in \cite{ChiapporiMcCannPass17} which apply to semi-discrete and continuous monopolist's problems.  It then provides general conditions under which solutions to multi-to one-dimensional screening problems satisfy nestedness, for both continuous and discrete sets of products.  This leads to a relatively simple general characterization of solutions, from which many examples can be solved explicitly, while others can be  solved numerically in a very efficient way.  A uniqueness result is also established (in Appendix~\ref{uniqueness}).  While nestedness of solutions is proven under various simplifying assumptions, including linearity in types of preference functions, and a two-dimensional type space, we believe that similar results are likely to hold in other situations as well. This is a natural direction for future work.

\appendix
\section{Connection to optimal transport}\label{connection to ot}

In this part, we present the optimal transport problem and its connection with the monopolist's problem. Let $\mu$ and $\nu$ be probability measures on bounded domains $X \subset \mathbb{R}^m$,  $Y\subset \mathbb{R}^n$  respectively, and let $b\in C(X\times Y)$ be the surplus function. Then the Monge-Kantorovich optimal transport problem is 
to find a measure $\gamma\in \Gamma(\mu,\nu)$ maximizing
\begin{equation}\label{eqn: OT primal}\tag{KP}KP:\,=\,\max_{\gamma\in\Gamma(\mu,\nu)}\int_{X\times Y} b(x,y)d\gamma(x,y),
\end{equation}
where $\Gamma(\mu,\nu)$ is the set of probability measures on $X\times Y$ with $\mu$ as the first marginal and $\nu$ as the second marginal, i.e 
\[\int_{A\times Y}d\gamma(x,y)=\mu(A) \text{ and }\int_{X\times B}d\gamma(x,y)=\nu(B)\]
for all measurable sets $A\subseteq X$ and  $B\subseteq Y.$ 

When an optimizer $\gamma$ vanishes outside $Graph(T)$, where $T\,:\,X\,\to Y,$ we call $T$ an \emph{optimal map}. In this case, $T$ satisfies
\[\nu(B)=T_{\#}\mu(B)=\mu[T^{-1}(B)]\] for all measurable sets $B\subseteq Y$  and we say $\nu$ is the push-forward of $\mu$ through $T.$

A powerful tool for understanding the Kantorovich problem is the dual linear program
\begin{equation}\label{eqn: OT dual}\tag{DP}
KP^*\,:=\, \inf_{(u,v)\in \mathcal{V}}\int_X u(x)d\mu(x)+\int_Y v(y)d\nu(y),
\end{equation}
where $\mathcal{V}$ is the set of payoff functions $(u,v)\in L^1(\mu)\times L^1(\nu)$ satisfying the inequality
\[u(x)+v(y)-b(x,y)\ge0\] on $X\times Y.$ 

It is well known that   $K\,P=K\,P^*$, solutions to both problems exist, and the optimal plan $\gamma$ in \eqref{eqn: OT primal} vanishes outside the zero set of the function $u+v-b$ where $(u,v)$ solve \eqref{eqn: OT dual} \cite{Santambrogio}. It is also known that the optimizers $(u,v)$ are $b$-convex conjugates, 
meaning that \[u(x)=\max_{y\in Y}b(x,y)-v(y)\text{ and }
v(y)=\max_{x\in X}b(x,y)-u(x).\]

In what follows, assume that $Y$ is one-dimensional $(n=1)$. We assume that the mixed second order derivative $D_x\Big(\frac{\partial b}{\partial y}(x,y)\Big)\neq 0$ for all $(x,y)\in X\times Y$ which implies by the Implicit Function Theorem that $\Big[\frac{\partial b}{\partial y}(\cdot,y)\Big]^{-1}(k)$ is of dimension $m-1$ for each constant $k \in \frac{\partial b}{\partial y}(X,y).$ To define the notion of nestedness introduced in \cite{ChiapporiMcCannPass17}, we start by  defining the following levels:

Assuming $\mu(A)>0$ for all nonempty open sets $A\subseteq X$ and that $\mu$ does not charge any $X_=(y,k)$ (that is, $\mu(X_=(y,k))=0$ for all $y \in Y$ and $k \in \mathbb{R}$), we define $k_+(y)$ and $k_-(y)$ such that 
\[\mu(X_{\le}(y,k_+(y)))=\nu((-\infty,y]) \text{ and } \mu(X_{<}(y,k_-(y)))=\nu((-\infty,y)).\]




\begin{definition}\label{nested}
    We say the optimal transport problem $(\mu,\nu,b)$ is nested if
    \[\text{for all } y_0,y_1 \text{ such that } y_0<y_1, \nu((y_0,y_1))>0\implies X_{\le}(y_0,\,k_+(y_0))\subset X_{<}(y_1,\,k_-(y_1)).\] 
\end{definition}\footnote{ The definition here actually differs slightly from the one in \cite{ChiapporiMcCannPass17}; in \cite{ChiapporiMcCannPass17}, it was assumed that the target measure $\nu$ is non-atomic.  We do not wish to make that assumption here, as we will connect the problem to the monopolist's problem in Proposition \ref{OT to PA} below; in that correspondence, the target measure $\nu$ corresponds to the distribution of goods produced by the monopolist.  This is endogenous, and may well contain atoms.}
In much of what follows, we will specialize to the case $b(x,y)=x\cdot z(y)$ where $z(y)$ parametrizes a one-dimensional curve; in this case, we will sometimes suppress $b$ and simply write that the OT problem $(\mu,\nu)$ is nested.   For general $b$, Chiappori-McCann-Pass prove that if the problem $(\mu,\nu,b)$ is nested, then the optimal map admits the following simple characterization: every $x \in X_=(y,k)$ for exactly one $y$ and some $k \in [k_-(y),k_+(y)]$, and the optimizer maps $x$ to this $y$ \cite{ChiapporiMcCannPass17}. Note that whenever $y$ is not an atom, $\nu(\{y\})=0$, we have that $k_-(y)=k_+(y)$, and for such $y$ we will sometimes denote this common value simply as $k(y)$.

Below, we establish an analogous result when the target measure $\nu$ is discrete.
\subsection{Semi-discrete optimal transport}
Consider the semi-discrete optimal transport problem where $\mu$ is a probability measure on $X\subset \mathbb{R}^m$ such that $\mu(A)>0$ for all nonempty open sets $A\subseteq X$, and $\nu=\sum_{i=0}^N \nu_i \delta_{y_i}$ is a probability measure on a finite $Y=Y_N=\{y_0,y_1,...,y_N\}$. In what follows we assume that $b$ satisfies $D_xb(x,y_i) -D_xb(x,y_{i-1}) \neq 0$ for all $x\in X$ and $0<i\le N.$ Using the Implicit Function Theorem on the equation $b(x,y_i)-b(x,y_{i-1})=k,$ we get that the preimage $[b(\cdot,y_i)-b(\cdot,y_{i-1})]^{-1}(k)$ is of dimension $m-1$ for each constant $k \in [b(\cdot,y_i)-b(\cdot,y_{i-1})](X).$   In this setting, the optimal plan $\gamma$ between $\mu$ and $\nu$ induces subregions $X_i$ such that all $x \in X_i$ are mapped to $y_i$. These regions can be described in terms of a potential function $v : Y_N \to \mathbb{R}$ that solves the dual problem \eqref{eqn: OT dual}. More precisely, for each $i$ we define
\[
X_i = \left\{ x \in X \,:\, b(x, y_i) - v(y_i) > b(x, y_j) - v(y_j) \quad \text{for all } j \neq i \right\}.
\]
Letting $u : X \to \mathbb{R}$ be the associated dual function defined by
\[
u(x) = \max_{0\le j\le N}  b(x, y_j) - v(y_j) ,
\]
we see that for all $x \in X_i$, the maximum is achieved uniquely at index $i$, so that $u(x) = b(x, y_i) - v(y_i)$. Hence, the boundary $\overline{X_i}\cap \overline{X_j} = \{x: u(x)=b(x,y_i)-v(y_i)=b(x,y_j)-v(y_j)\}$ between $X_i$ and $X_j$ is the set of \emph{indifference points}; each such agent has their utility maximized by both $y_j$ and $y_i$. When $m=2$, we will sometimes refer to $\overline{X_i}\cap \overline{X_j}$ as an \emph{indifference curve}.   We define two regions $X_i$ and $X_j$ to be \emph{adjacent} if their indifference set $\overline{X_i} \cap \overline{X_j}$ has positive $(m-1)$-dimensional Hausdorff measure.
  
When $Y$ is discrete, we define nestedness of the semi-discrete optimal transport as follows:
\begin{definition}
    We say the optimal transport problem $(\mu,\nu,b)$ is discretely nested if 
    \[ \nu(\{y_{k}:\, i< k\le j\})>0\,\implies\,X_{\le}^N(y_{i},k^N(y_{i}))\subset X_{<}^N(y_{j},k^N(y_{j})),\] 
     for all $i<j$  where $k^N(y_{r})$  satisfies $\mu(X_{\le}^N(y_{r},k^N(y_{r})))=\nu(\{y_{p}:\, 0\le p\le r \}).$
\end{definition}

\begin{remark}
    If $\mu(X^N_{=}(y_i, k))=0$ for all $y_i$ and $k$, then by the continuity of $k\mapsto h(y_i,k):=\mu(X_{\le}^N(y_i,k))-\nu(\{y_{p}:\, 0\le p\le i \})$ and as $k\mapsto h(y_i,k)$ goes monotonically from $-\nu(\{y_{p}:\, 0\le p\le i \}) \leq 0$ to $1-\nu(\{y_{p}:\, 0\le p\le i \})\geq 0,$ by the Intermediate Value Theorem, there exists $k^N(y_i)$ such that $h(y_i,k^N(y_i))=0.$ Since $\mu$ assigns positive measure to every nonempty open subset of $X$, we get the uniqueness of $k^N(y_i).$ 

\end{remark}

The following result provides a characterization of the solution of the discretely nested optimal transport problems.
\begin{theorem}\label{nestchar}
Assume that the optimal transport problem $(\mu, \nu, b)$ is discretely nested and $\mu$ does not charge any $X^N_{=}(y_i, k)$. Then, setting $X_0=X^N_{<}(y_0, k^N(y_0)),$ $X_N=X\setminus X^N_{\le}(y_{N-1}, k^N(y_{N-1})),$ and $X_i = X^N_{<}(y_i, k^N(y_i)) \setminus X^N_{\leq}(y_{i-1}, k^N(y_{i-1}))$ for all $0<i<N$, the potentials $(u,v)$ defined as $u(x) = b(x, y_i) - v(y_i)$ for all $x \in X_i$,  such that 
\[
v(y_i) = \sum_{k=0}^{i-1}(b(a_k, y_{k+1}) - b(a_k, y_k))
\]
with $v(y_0) = 0$, solve the dual problem \eqref{eqn: OT dual} for any $a_k \in X^N_{=}(y_k, k^N(y_k))$. Furthermore, the mapping $T$ sending all $x \in X_i$ to $y_i$ for each $i=0,1,...,N$ is an optimal map.
\end{theorem}

\begin{corollary}\label{do not intersect}
    If $(\mu,\nu,b)$ is discretely nested, then no indifference curves of the solution intersect in $X.$
\end{corollary}

\textbf{Connection between OT and the monopolist's problem:}
Returning to the monopolist's problem, given a pricing function  $v$ and corresponding competitor $u \in \mathcal{U}$ in  \eqref{PA}, define $\nu =y^*_\#\mu$, representing the distribution of products sold.  Then it is well known that $y^*$ is an optimal map for the optimal transport problem \eqref{eqn: OT primal} with surplus $b$ and marginals $\mu$ and $\nu$,  while $u$ and $v$ solve its dual \eqref{eqn: OT dual}~\cite{FigalliKimMcCann11}.
Thus, any feasible competitor in, and, in particular, any solution to, the monopolist's problem induces a solution to an optimal transport problem.

\begin{prop}\label{OT to PA discrete}
     Let $u\in\mathcal{U}$ be a solution of the monopolist's problem \eqref{PA} with data $(\mu,Y_N,c)$. If the optimal transport problem $(\mu,y^*_\#\mu,b)$ is discretely nested and $y^*_\#\mu(\{y_i\})>0$ for all $\underline{M}\le i\le \overline{M},$ and \(y^*_\#\mu(\{y_i\}) = 0\) otherwise, for some $0\le \underline{M}\le \overline{M}\le N$, then $u$ is a discretely nested solution of \eqref{PA}.
\end{prop}
\begin{proof}
    Let $(u,v)$ be the solution of the dual problem \eqref{eqn: OT dual} of $(\mu,y^*_\#\mu,b).$  By Theorem \ref{nestchar}, we conclude that $k^N(y_i)=v(y_{i+1})-v(y_i)$ for all $0\le i<N.$ When  $\underline{M}\le i<\overline{M},$ we get $\nu(\{y_{k}:\, i< k\le j\})>0$ for all $j>i,$ and as $(\mu,y^*_\#\mu,b)$ is discretely nested, we get $X_{\le}^N(y_{i},v(y_{i+1})-v(y_i))\subset X_{<}^N(y_{j},v(y_{j+1})-v(y_j)).$ When $i\ge \overline{M},$ we get $X=X_{\le}^N(y_{i},v(y_{i+1})-v(y_i))= X_{<}^N(y_{j},v(y_{j+1})-v(y_j)).$ Similarly, we get $\emptyset=X_{\le}^N(y_{i},v(y_{i+1})-v(y_i))= X_{<}^N(y_{j},v(y_{j+1})-v(y_j))$ whenever $i<j<\underline{M}.$ Thus, $u$ is discretely nested solution of \eqref{PA}.
\end{proof}

\begin{prop}\label{OT to PA}
    Let $u\in\mathcal{U}$ be a solution of the monopolist's problem \eqref{PA} with data $(\mu,Y,c)$. If the optimal transport problem $(\mu,y^*_\#\mu,b)$ is nested and the support of $y^*_\#\mu$ is connected, then $u$ is a nested solution of \eqref{PA}.
\end{prop}

\begin{proof}

Assume that $(\mu,y^*_\#\mu,b)$ is nested where $y^*$ is defined as above.  Note that $(u,v)$  is the solution of the dual problem \eqref{eqn: OT dual} of $(\mu,y^*_\#\mu,b)$  where $v(y)=\max_{x\in X}\{b(x,y)-u(x)\}.$    Following \cite{ChiapporiMcCannPass17}, we introduce the \emph{$b$-subdifferential} of $v$ at $y$, defined by 
\(
\partial_b v(y) := \Bigl\{x \in X : \, b(x,y)-v(y) \;\ge\; b(x,y')-v(y') \quad \text{for all } y'\in Y \Bigr\}.
\)  
It follows that the matching $y^*$ which sends
$x \in X_\le(y,k_+(y)) \setminus X_<(y,k_-(y)) \;=\; \partial_b v(y)$
to $y$ is the Monge solution of $(\mu,y^*_\#\mu,b)$. From the first  order optimality condition of optimal transport, we get that $\frac{\partial b}{\partial y}(x,y^*(x))\in \partial v(y^*(x))$ where $\partial v(y)$ is the subdifferential of $v$ at $y.$ Let $k\in[k_-(y),k_+(y)],$ then there exists $x_k\in  X_\le(y,k_+(y))\setminus X_<(y,k_-(y))$ such that $\frac{\partial b}{\partial y}(x_k,y)=k.$ Thus, $[k_-(y),k_+(y)] \subset \partial v(y)$ for all $y\in Y.$ 

Since 
\[
\partial v(y)
= \operatorname{con}\left\{
\frac{\partial b}{\partial y}(x,y) :\, x \in \partial_b v(y)
\right\}
\]
(see Theorem~10.31 of~\cite{RockafellarWets1998}), 
where $\operatorname{con}(\cdot)$ denotes the convex hull, 
it follows that
\[
\partial v(y)
= \bigl[
\min_{x \in \partial_b v(y)} \tfrac{\partial b}{\partial y}(x,y),\,
\max_{x \in \partial_b v(y)} \tfrac{\partial b}{\partial y}(x,y)
\bigr].
\]
we deduce that $[k_-(y),k_+(y)] = \partial v(y)=[v_-'(y),v'_+(y)].$ 

As the support of \(y^*_\#\mu\) is connected, we have 
\(y^*_\#\mu\) supported on \([\underline{s}, \overline{s}]\).  
Whenever \(y \le \underline{s}\), we obtain 
\(X_{<}(y, k_-(y)) = X_{\le}(y', k_+(y')) = \emptyset\) 
for all \(y' < y\) by the definition of \(k_{\pm}\).  
Similarly, we get \(X_{<}(y, k_-(y)) = X_{\le}(y', k_+(y')) = X\) 
whenever \(\overline{s} \le y' < y\).  
Now, when \(\underline{s} \le y' < \overline{s}\), we have \(\nu((y',y)) > 0\) for all $y'<y$, 
which implies \(X_{\le}(y', v'_+(y')) \subset X_{<}(y, v'_-(y))\), 
and therefore \(u\) is nested.  

\end{proof}

\section{Proofs}
Proof of Theorem \ref{nestchar}:
\begin{proof}
It is clear from the construction that the mapping $T$, which maps each $X_i$ to $y_i$, pushes $\mu$ forward to $\nu$. The other conclusions will follow from Kantorovich duality if we can show $u(x) + v(y) \geq b(x, y)$ for all $x \in X, y \in Y_N$, with equality when $T(x) = y$.

Set $v(y_0):=v_0 = 0$ and $v(y_i):=v_i = \sum_{k=0}^{i-1}(b(a_k, y_{k+1}) - b(a_k, y_k))$ for $i \geq 1$; note that this is well-defined since $x \mapsto b(x, y_{k+1}) - b(x, y_k)$ is constant along $X^N_{=}(y_k, k^N(y_k))$. We need to show that
\[
u_i(x) := b(x, y_i) - v_i \geq b(x, y_j) - v_j := u_j(x)
\]
for all $j$ when $x \in X_i$. Note that for these $v_i$, we have
\[
b(x, y_i) - v_i = b(x, y_{i+1}) - v_{i+1}
\]
along $X^N_{=}(y_i, k^N(y_i))$, and therefore $u_i = b(x, y_i) - v_i > b(x, y_{i-1}) - v_{i-1} = u_{i-1}$ throughout $X_i \subseteq X^N_{\leq}(y_i, k^N(y_i)) \setminus X^N_{\leq}(y_{i-1}, k^N(y_{i-1}))$. Now, the discrete nestedness condition also implies that $X^N_{\leq}(y_{i-1}, k^N(y_{i-1})) \subset X^N_{\leq}(y_i, k^N(y_i))$ and $X_i \subset X \setminus X^N_{\leq}(y_{i-1}, k^N(y_{i-1}))$, where $u_{i-1} > u_{i-2}$. Hence, in $X_i$ we have
\[
u_i > u_{i-2}.
\]
Continuing in this way, we can show that throughout $X_i$,
\[
u_i \geq u_{i-1} > u_{i-2} > \dots > u_j
\]
for all $j < i$. A similar argument shows $u_i \geq u_j$ for $j > i$, completing the proof.
\end{proof}
We next prove Corollary \ref{do not intersect}:
\begin{proof}
    From Theorem \ref{nestchar}, we know that the indifference curves of the solutions are the level sets $X_=(y_i,k^N(y_i)).$ For all $i<j,$ we have $X_=(y_i,k^N(y_i))\subseteq X_\le(y_i,k^N(y_i))\subset X_<(y_j,k^N(y_j))=X_\le(y_j,k^N(y_j))\setminus X_=(y_j,k^N(y_j)),$ which implies $X_=(y_i,k^N(y_i))\cap X_=(y_j,k^N(y_j))=\emptyset$ completing the proof.
\end{proof}
We next prove the assertions in Example \ref{ex1}:

\begin{proof}
    It is easy to check that $c$ and $F$ satisfy conditions \ref{condition2} and \ref{condition3}. We will prove condition \ref{condition4} is satisfied. Let $F(y)=Ay^2,$ then 
    \[\frac{1}{2}\frac{F(y_{i+1})-F(y_i)}{y_{i+1}-y_i}\Big(2+\frac{\frac{F(y_{i+1})-F(y_i)}{y_{i+1}-y_i}}{\frac{F(y_{i})-F(y_{i-1})}{y_{i}-y_{i-1}}}\Big)=A\frac{y_{i+1}+y_i}{2}\Big(2+\frac{y_{i+1}+y_i}{y_{i}+y_{i-1}}\Big).\] It is sufficient to prove that $A\frac{y_{i+1}+y_i}{2}\Big(2+\frac{y_{i+1}+y_i}{y_{i}+y_{i-1}}\Big)<1.$ We have 
    \[\begin{array}{ll}
    A\frac{y_{i+1}+y_i}{2}\Big(2+\frac{y_{i+1}+y_i}{y_{i}+y_{i-1}}\Big)&\le\frac{y_{i+1}+y_i}{6}\Big(2+\frac{y_{i+1}+y_i}{y_{i}+y_{i-1}}\Big)\vspace{3pt}\\
    &\le\frac{2i+1}{6N}\Big(3+\frac{y_{i+1}-y_{i-1}}{y_{i}+y_{i-1}}\Big)\vspace{3pt}\\
    &\le\frac{2i+1}{6N}\Big(3+\frac{2\cos(\theta_{i-1})}{(2i-1)\cos(\theta_{i-1})}\Big),
    
    \end{array}\]
    where $\theta_r$ is the angle between the vector $(y_{r+1}-y_r,F(y_{r+1})-F(y_r))=(\frac{\cos(\theta_r)}{N},\frac{\sin(\theta_r)}{N})$ and the $x_1-$axis. Note that the last inequality comes from the fact that $F$ is convex and then $y_{i+1}-y_{i-1}=\frac{1}{N}(\cos(\theta_{i-1})+\cos(\theta_i))\le\frac{2\cos(\theta_{i-1})}{N}.$ Also, we have $(2i-1)\cos(\theta_{i-1})\le y_i+y_{i-1},$ as $\cos(\theta_r)$ is decreasing in $r$ from the convexity of $F.$ Hence,
    
    \[\begin{array}{ll}
    A\frac{y_{i+1}+y_i}{2}\Big(2+\frac{y_{i+1}+y_i}{y_{i}+y_{i-1}}\Big)&\le\frac{2i+1}{6N}\Big(3+\frac{2}{2i-1}\Big)=\frac{1}{6N}\Big(3(2i+1)+\frac{4i+2}{2i-1}\Big).
    
    \end{array}\]
    When $i=0,1$ the claim is satisfied.  Since $g(s)=3(2s+1)+\frac{4s+2}{2s-1}$ is increasing for all $s\ge2,$ we get $g(i)\le g(N-1),$ and then
    \[\begin{array}{ll}
    A\frac{y_{i+1}+y_i}{2}\Big(2+\frac{y_{i+1}+y_i}{y_{i}+y_{i-1}}\Big)&\le\frac{1}{6N}\Big(3(2(N-1)+1)+\frac{4(N-1)+2}{2(N-1)-1}\Big)=\frac{12N^2-18N-(2N-7)}{12N^2-18N}<1
    
    \end{array}\]
    when $N>3$ which completes the proof.
\end{proof}
Next we prove Lemma \ref{partition}:
\begin{proof}
    Assume there exists $k$ such that $\overline{X_{i_k}}\cap (([0,1]\times\{0\}) \cup(\{1\}\times[0,1]))=\emptyset.$ Since $z_j-z_i$ is in the direction of a line with slope $\overline{s},$ where  $0<\overline{s}<F'(\Tilde{y}),$   the boundary of $X_{i_k}$ in the interior of $X$ consists of segments with slope $-\frac{1}{s}< 0$ for some $0<s<F'(\Tilde{y}).$ Hence, the region $\overline{X_{i_k}}$ contains a point $e=(e_1,e_2)$ with $e_2=\min\{x_2:\, (x_1,x_2)\in \overline{X_{i_k}},\,  0<x_1<1 \}$ which is contained in $\overline{X}.$ This implies that $e$ is the intersection  of two boundary segments of $X_{i_k}$ that lie on  the graphs of $L_1(x_2)=e_1-s_1(x_2-e_2)$ and $L_2(x_2)=e_1-s_2(x_2-e_2)$ such that  $-\frac{1}{s_1}>-\frac{1}{s_2}$ and  $L_1(x_2)<L_2(x_2)$ when $x_2>e_2.$  
   Then, $s_1>s_2$  and since $F$ is an increasing convex function, we deduce that $s_1$ corresponds to the slope of $z_i-z_j$ and  $s_2$ to the slope of $z_i-z_p$ where $j>p,$ which means that part of the graph of $L_1$ is the indifference segment between $X_i$ and $X_j$ and part of the graph of $L_2$ is the indifference segment between $X_i$ and $X_p.$ However, for small enough $\varepsilon>0$ the line $x_2=e_2+\varepsilon$ intersects the two segments and since $u$ is convex, $u(x_1,e_2+\varepsilon)$  is convex in $x_1$. Then,  $\frac{\partial}{\partial x_1}(u(x_1,e_2+\varepsilon))$  is increasing. But, since $L_1(x_2)<L_2(x_2),$ $\frac{\partial}{\partial x_1}(u(x_1,e_2+\varepsilon))$ starts from $y_j$ and decreases to $y_p$ between $L_1(e_2+\varepsilon)-\delta$ and $L_2(e_2+\varepsilon)+\delta$ for small enough $\delta>0,$  which is a contradiction.

For the second part, we extend $u$ by continuity to $\overline{X}.$ Since $u(x_1,0)$ and $u(1,x_2)$ are convex functions, $\frac{\partial u}{\partial x_1}(x_1,0)$ is increasing when $0\le x_1\le1$ from $y_{i_0}$ to $y_{i_{k0}}$ and $\frac{\partial u}{\partial x_2}(1,x_2)$ is increasing from $F(y_{i_{k0}})$ to  $F(y_{i_p})$ for some $z_{i_{k_0}},z_{i_{p}}\in (z_{i_{k}}).$ Hence, $\overline{X}_{i_{k+1}}\cap\overline{X}_{i_k}\cap \Big(([0,1]\times\{0\}) \cup(\{1\}\times[0,1])\Big)\neq \emptyset,$ using the first part. 
\end{proof}
Theorem \ref{between} says that, roughly stated, if, given some pricing plan,  two goods are purchased and one between them is not, the plan cannot be optimal.  The next three lemmas establish this fact in different cases, depending on where the indifference line between consumers choosing these goods intersects the boundary.
\begin{lem}\label{between1}
   Let $ u \in \mathcal{U} $, and assume that $ c $ and $ F $ satisfy condition \ref{condition2}. Suppose there exist indices $ k < j - 1 $ such that $ \mu(X_k) > 0 $ and $ \mu(X_j) > 0 $, while $ \mu(X_i) = 0 $ for all $ k < i < j $. Additionally, assume that $ X_k $ and $ X_j $ are adjacent, and the set of indifference points between $ X_k $ and $ X_j $ intersects the segment $ (0,1) \times \{0\} $. Under these conditions, $ u $ cannot be a solution to the monopolist's problem.

\end{lem}

\begin{proof}

          Suppose that $k<i<j$ and $\mu(X_i)=0.$  We will show that lowering the price of good $y_i$ leads to increased profits.  

          Lowering the price of $z_i$ by $\delta=a\cdot(z_j-z_i)-(v(z_j)-v(z_i))$ results in a new region of customers  $X_i^a$, with positive mass, choosing $z_i$ under the lowered price, where  $a\cdot z_j-v(z_j)=a\cdot z_i-v(z_i)+\delta=a\cdot z_k-v(z_k)$.  That is, $a=(a_1,\,a_2)\in X$ is the point on the original indifference segment between $X_k$ and $X_j$ which is also the intersection of the indifference segment between $X_k^a$ and $X_i^a$ and the indifference segment between $X_i^a$ and $X_j^a$, where $X_p^a$ is the set of customers choosing $z_p$ after lowering the price of $z_i$ (note that by Lemma \ref{partition},  $X_i^a$  is adjacent to $X_k^a$ and $X_j^a$ ). Let $u_a$ be the new payoff function.

          Since $u_a=u$ everywhere except on $ X_i^a,$ and since  $u_a(x)=(x-a)\cdot z_i+u(a)$ on $X_i^a=(X_k\cap X_i^a)\cup(X_j\cap X_i^a),$ and $u(x)=(x-a)\cdot z_k+u(a)$ on $ X_k\cap X_i^a$ and $u(x)=(x-a)\cdot z_j+u(a)$ on $ X_j\cap X_i^a,$ we can  evaluate the difference in profit in terms of $a$ for small enough $a_2$ as follows
        \[\begin{array}{ll}
            
       \mathcal{P}(u_a)-\mathcal{P}(u)&=
       \int_{X_i^a}(x\cdot(Du_a-Du)-(u_a-u)-(c(Du_a)-c(Du)))f(x)\,dx\vspace{3pt}\\
       &= -(a\cdot(z_j-z_i)-(c(z_j)-c(z_i))\mu( X_j\cap X_i^a)\vspace{3pt}\\
       &\quad +(a\cdot(z_i-z_k)-(c(z_i)-c(z_k))\mu( X_k\cap X_i^a)\vspace{3pt}\\
&= (-a_2(F(y_j)-F(y_i))-a_1(y_j-y_i)+(c(z_j)-c(z_i))\mu( X_j\cap X_i^a)\vspace{3pt}\\
 & \quad +(a_2(F(y_i)-F(y_k))+a_1(y_i-y_k)
      -(c(z_i)-c(z_k))\mu( X_k\cap X_i^a).
        \end{array}\]
     
      In the above expression,  the term $-a_2((F(y_j)-F(y_i))\mu( X_j\cap X_i^a)-(F(y_i)-F(y_k))\mu( X_k\cap X_i^a))$ is of higher order in $a_2$ than the other terms, so for small enough $a_2,$ to study the sign of the difference, it is sufficient to study the sign of
      \begin{equation}\label{c eq}\begin{array}{ll}
      &a_1(-(y_j-y_i)\mu( X_j\cap X_i^a)+(y_i-y_k)\mu( X_k\cap X_i^a))\\
      &+(c(z_j)-c(z_i))\mu( X_j\cap X_i^a)-(c(z_i)-c(z_k))\mu( X_k\cap X_i^a).\end{array}\end{equation}

We have \[\begin{array}{ll}
I_1&=-(y_j-y_i)\mu( X_j\cap X_i^a)+(y_i-y_k)\mu( X_k\cap X_i^a)\vspace{10pt}\\
&=\int_0^{a_2}\Big(\int_{r_i}^{r_j}-(y_j-y_i)f(x)\,dx_1+\int_{r_k}^{r_i}(y_i-y_k)f(x)\,dx_1\Big)\,dx_2,
\end{array}\]
 where $r_i=\frac{F(y_j)-F(y_k)}{y_j-y_k}(a_2-x_2)+a_1,$ $r_j=\frac{F(y_j)-F(y_i)}{y_j-y_i}(a_2-x_2)+a_1,$ and $r_k=\frac{F(y_i)-F(y_k)}{y_i-y_k}(a_2-x_2)+a_1.$     

      Changing the variable in the first integral we get 

      \[I_1=\int_0^{a_2}\Big(\int_{r_k}^{r_i}-(y_j-y_i)\frac{r_j-r_i}{r_i-r_k}f(\beta(x_1),x_2)+(y_i-y_k)f(x_1,x_2)\,dx_1\Big)\,dx_2,\]
where $\beta(x_1)=\frac{r_j-r_i}{r_i-r_k}(x_1-r_k)+r_i.$ 
This integral is equal to

 \[\begin{array}{ll}
 I_1&=\int_0^{a_2}\Big(\int_{r_k}^{r_i}-(y_j-y_i)\frac{r_j-r_i}{r_i-r_k}(f(\beta(x_1),x_2)-f(x_1,x_2))\vspace{10pt}\\
 &\quad +\Big((y_i-y_k)-(y_j-y_i)\frac{r_j-r_i}{r_i-r_k}\Big)f(x_1,x_2)\,dx_1\Big)\,dx_2.
 \end{array}\]

Now by a straightforward calculation, we get 
\[(y_i-y_k)-(y_j-y_i)\frac{r_j-r_i}{r_i-r_k}=\frac{(r_i-r_k)(y_i-y_k)-(y_j-y_i)(r_j-r_i)}{r_i-r_k}=0.\]

Then, 
\[\begin{array}{ll}

I_1 &=(y_i-y_k)\int_0^{a_2}\int_{r_k}^{r_i}\Big(-f(\beta(x_1),x_2)+f(x_1,x_2)\,dx_1\Big)\,dx_2\vspace{3pt}\\
&\ge -(y_i-y_k)\int_0^{a_2}\int_{r_k}^{r_i}\|f_{x_1}\|_\infty(\beta(x_1)-x_1)\,dx_1\,dx_2>K_1a_2^3

\end{array}\]
as $\beta(x_1)-x_1$ is linear in $a_2$ and we are integrating over a triangle with an area that is quadratic in $a_2.$

Now, for the second part of (\ref{c eq}), after changing the variable in the first term we get
\[\begin{array}{ll}
I_2&=(c(z_j)-c(z_i))\mu( X_j\cap X_i^a)-(c(z_i)-c(z_k))\mu( X_k\cap X_i^a)\vspace{3pt}\\

&=\int_0^{a_2}\Big(\int_{r_k}^{r_i}(c(z_j)-c(z_i))\frac{r_j-r_i}{r_i-r_k}(f(\beta(x_1),x_2)-f(x_1,x_2))+((c(z_j)-c(z_i))\frac{r_j-r_i}{r_i-r_k}\vspace{3pt}\\
&\quad -(c(z_i)-c(z_k))) f(x_1,x_2)\,dx_1\Big)\,dx_2,\vspace{3pt}\\
&\geq K_2a_2^3+\int_0^{a_2}\frac{\alpha}{r_i-r_k}((c(z_j)-c(z_i))(r_j-r_i)-(c(z_i)-c(z_k))(r_i-r_k))\int_{r_k}^{r_i}\,dx_1dx_2\vspace{3pt}\\
&=K_2a_2^3+K_3a_2^2,
\end{array}\] 
where the term inside the integral in the last line,
\[\begin{array}{ll}
&\frac{(c(z_j)-c(z_i))(r_j-r_i)-(c(z_i)-c(z_k))(r_i-r_k)}{r_i-r_k}\vspace{3pt}\\
&=\frac{(a_2-x_2)\Big((c(z_j)-c(z_i))\Big(\frac{F(y_j)-F(y_i)}{y_j-y_i}-\frac{F(y_j)-F(y_k)}{y_j-y_k}\Big)-(c(z_i)-c(z_k))\Big(\frac{F(y_j)-F(y_k)}{y_j-y_k}-\frac{F(y_i)-F(y_k)}{y_i-y_k}\Big)\Big)}{(a_2-x_2)\Big(\frac{F(y_j)-F(y_k)}{y_j-y_k}-\frac{F(y_i)-F(y_k)}{y_i-y_k}\Big)}.\\
&=\frac{\Big((c(z_j)-c(z_i))\Big(\frac{F(y_j)-F(y_i)}{y_j-y_i}-\frac{F(y_j)-F(y_k)}{y_j-y_k}\Big)-(c(z_i)-c(z_k))\Big(\frac{F(y_j)-F(y_k)}{y_j-y_k}-\frac{F(y_i)-F(y_k)}{y_i-y_k}\Big)\Big)}{\Big(\frac{F(y_j)-F(y_k)}{y_j-y_k}-\frac{F(y_i)-F(y_k)}{y_i-y_k}\Big)},
\end{array}\]
is constant and  
\[\begin{array}{ll}
K_3=&\hspace{-0.35cm}\frac{\alpha}{2}\Big((c(z_j)-c(z_i))\Big(\frac{F(y_j)-F(y_i)}{y_j-y_i}-\frac{F(y_j)-F(y_k)}{y_j-y_k}\Big)\vspace{3pt}\\
&-(c(z_i)-c(z_k))\Big(\frac{F(y_j)-F(y_k)}{y_j-y_k}-\frac{F(y_i)-F(y_k)}{y_i-y_k}\Big)\Big).
\end{array}\] 
  We claim that $K_3$ is positive. We have 
\begin{equation}\label{H1}\hspace{-6pt}\begin{array}{ll}
&\hspace{-10pt}(c(z_j)-c(z_i))\Big(\frac{F(y_j)-F(y_i)}{y_j-y_i}-\frac{F(y_j)-F(y_k)}{y_j-y_k}\Big)-(c(z_i)-c(z_k))\Big(\frac{F(y_j)-F(y_k)}{y_j-y_k}-\frac{F(y_i)-F(y_k)}{y_i-y_k}\Big)\vspace{3pt}\\
=&\hspace{-8pt}(c(z_j)-c(z_i))\frac{(F(y_j)-F(y_i))(y_j-y_k)-(F(y_j)-F(y_i)+F(y_i)-F(y_k))(y_j-y_i)}{(y_j-y_i)(y_j-y_k)}\vspace{3pt}\\
&\hspace{-8pt}-(c(z_i)-c(z_k))\frac{(F(y_j)-F(y_i)+F(y_i)-F(y_k))(y_i-y_k)-(F(y_i)-F(y_k))(y_j-y_k)}{(y_i-y_k)(y_j-y_k)}\vspace{3pt}\\
=&\hspace{-8pt}\frac{(F(y_j)-F(y_i))(y_i-y_k)-(F(y_i)-F(y_k))(y_j-y_i)}{(y_j-y_k)}\Big(\frac{c(z_j)-c(z_i)}{y_j-y_i}-\frac{c(z_i)-c(z_k)}{y_i-y_k}\Big)>0,
\end{array}
\end{equation}
 since $\frac{F(y_j)-F(y_i)}{y_j-y_i}>\frac{F(y_i)-F(y_k)}{y_i-y_k}$ and by \ref{condition2} we get
$0<\frac{c(z_j)-c(z_i)}{F(y_j)-F(y_i)}-\frac{c(z_i)-c(z_k)}{F(y_i)-F(y_k)}
=\frac{c(z_j)-c(z_i)}{F'(\overline{y}_{ji})(y_j-y_i)}-\frac{c(z_i)-c(z_k)}{F'(\overline{y}_{ik})(y_i-y_k)}<\frac{1}{F'(\overline{y}_{ik})}\big(\frac{c(z_j)-c(z_i)}{y_j-y_i}-\frac{c(z_i)-c(z_k)}{y_i-y_k}\big)$
for some $y_j\ge\overline{y}_{ji}\ge y_i$ and $y_i\ge \overline{y}_{ik}\ge y_k,$ which proves our claim.

Hence, $\mathcal{P}(u_a)-\mathcal{P}(u)
\ge o(a_2^4)+K_1a_2^3+K_2a_2^3+K_3a_2^2,$ and from the order of the terms we conclude that expression \eqref{c eq} has the same sign as $K_3>0,$ for sufficiently small $a_2,$ which completes the proof.
\end{proof}

\begin{lem}\label{between2}
     Let $ u \in \mathcal{U} $, and assume that $ c $ and $ F $ satisfy condition \ref{condition2}. Suppose there exist indices $ k < j - 1 $ such that $ \mu(X_k) > 0 $ and $ \mu(X_j) > 0 $, while $ \mu(X_i) = 0 $ for all $ k < i < j $. Additionally, assume that $ X_k $ and $ X_j $ are adjacent, and the set of indifference points between $ X_k $ and $ X_j $ intersects the segment $\{1\}\times(0,1).$ Then, $u$ is not a solution of the monopolist's problem. 
\end{lem}
 \begin{proof}
Suppose that   $k<i<j$ and $\mu(X_i)=0.$ Let $u_a$ be the utility function defined as in Lemma \ref{between1}. We evaluate the difference in profit as follows
\[\hspace{-6pt}\begin{array}{ll} &\mathcal{P}(u_a)-\mathcal{P}(u) \vspace{3pt}\\
      =&\hspace{-6pt} (-a_2(F(y_j)-F(y_i))-a_1(y_j-y_i)+(c(z_j)-c(z_i))\mu( X_j\cap X_i^a)\vspace{3pt}\\
      &\hspace{-6pt} +(a_2(F(y_i)-F(y_k))+a_1(y_i-y_k)
      -(c(z_i)-c(z_k))\mu( X_k\cap X_i^a)\vspace{3pt}\\
  =&\hspace{-6pt}(-a_2(F(y_j)-F(y_i))+(1-a_1)(y_j-y_i)-(y_j-y_i)+(c(z_j)-c(z_i))\mu( X_j\cap X_i^a)\vspace{3pt}\\
      &\hspace{-6pt}+(a_2(F(y_i)-F(y_k))-(1-a_1)(y_i-y_k)+(y_i-y_k)-(c(z_i)-c(z_k))\mu( X_k\cap X_i^a)\vspace{3pt}\\
      =&\hspace{-6pt}(1-a_1)((y_j-y_i)\mu( X_j\cap X_i^a)-(y_i-y_k)\mu( X_k\cap X_i^a))\vspace{3pt}\\
      
      &\hspace{-6pt}+a_2(-(F(y_j)-F(y_i))\mu( X_j\cap X_i^a)+(F(y_i)-F(y_k))\mu( X_k\cap X_i^a))\vspace{3pt}\\
      &\hspace{-6pt}+((c(z_j)-c(z_i))-(y_j-y_i))\mu( X_j\cap X_i^a)-((c(z_i)-c(z_k))-(y_i-y_k))\mu( X_k\cap X_i^a).
      \end{array}\]
 As $a_1\to1,$ $\mathcal{P}(u_a)-\mathcal{P}(u)$ has the same sign as 
\[\hspace{-0.45cm}\begin{array}{ll}
&\quad a_2(-(F(y_j)-F(y_i))\mu( X_j\cap X_i^a)+(F(y_i)-F(y_k))\mu( X_k\cap X_i^a))\vspace{3pt}\\
&\quad +((c(z_j)-c(z_i))-(y_j-y_i))\mu( X_j\cap X_i^a)-((c(z_i)-c(z_k))-(y_i-y_k))\mu( X_k\cap X_i^a)\vspace{3pt}\\
&=a_2I_1+I_2
\end{array}\]
We have
\[\begin{array}{ll}
I_1&=-\int_{a_1}^1\int_{r_i}^{r_j}(F(y_j)-F(y_i))f(x)\,dx_2-\int_{r_k}^{r_i}(F(y_i)-F(y_k))f(x)\,dx_2\,dx_1\vspace{3pt}\\
&=-\int_{a_1}^1\int_{r_k}^{r_i}(F(y_j)-F(y_i))\frac{r_j-r_i}{r_i-r_k}(f(x_1,\beta(x_2))-f(x_1,x_2))\vspace{3pt}\\
&\quad +((F(y_j)-F(y_i))\frac{r_j-r_i}{r_i-r_k}-(F(y_i)-F(y_k))f(x_1,x_2)\,dx_2\,dx_1\vspace{3pt}\\
&\ge-\int_{a_1}^1\int_{r_k}^{r_i}(F(y_j)-F(y_i))\frac{r_j-r_i}{r_i-r_k}\|f_{x_2}\|_\infty(\beta(x_2)-x_2)\,dx_2\,dx_1\vspace{3pt}\\
&=K_1(1-a_1)^3
\end{array}\]
 where  $r_i=-\frac{y_j-y_k}{F(y_j)-F(y_k)}(x_1-a_1)+a_2,$  $r_j=-\frac{y_j-y_i}{F(y_j)-F(y_i)}(x_1-a_1)+a_2,$ and $r_k=-\frac{y_i-y_k}{F(y_i)-F(y_k)}(x_1-a_1)+a_2,$
and we  change the variable in the first integral and we get $\beta(x_2)=\frac{r_j-r_i}{r_i-r_k}(x_2-r_k)+r_i.$ Note that 
\[(F(y_j)-F(y_i))\frac{r_j-r_i}{r_i-r_k}-(F(y_i)-F(y_k))=0.\]
Moving to $I_2,$  we have 
\begin{equation}\label{c eq2}\hspace{-6pt}
\begin{array}{ll}

I_2\hspace{-8pt}&=\int_{a_1}^1\int_{r_i}^{r_j}((c(z_j)-c(z_i))-(y_j-y_i))f(x)\,dx_2\vspace{3pt}\\
&\quad -\int_{r_k}^{r_i}((c(z_i)-c(z_k))-(y_i-y_k))f(x)\,dx_2\,dx_1 \vspace{3pt}\\
&=\int_{a_1}^1\int_{r_k}^{r_i}((c(z_j)-c(z_i))-(y_j-y_i))\frac{r_j-r_i}{r_i-r_k}(f(x_1,\beta(x_2))-f(x_1,x_2))\vspace{3pt}\\
& \quad+((c(z_j)-c(z_i))-(y_j-y_i))\frac{r_j-r_i}{r_i-r_k}-((c(z_i)-c(z_k))-(y_i-y_k)))f(x_1,x_2)\,dx_2\,dx_1\vspace{3pt}\\
&\ge K_2(1-a_1)^3+\alpha((c(z_j)-c(z_i))-(y_j-y_i))\frac{r_j-r_i}{r_i-r_k}\vspace{3pt}\\
&\hspace{-8pt}\quad -((c(z_i)-c(z_k))-(y_i-y_k)))K_3(1-a_1)^2,
\end{array}
\end{equation}
where $K_3$ is a positive number . Hence, it is sufficient to study the sign of  
\[\frac{((c(z_j)-c(z_i))-(y_j-y_i))(r_j-r_i)-((c(z_i)-c(z_k))-(y_i-y_k))(r_i-r_k))}{r_i-r_k}.\]
We have 
\[\begin{array}{ll}
&\frac{(y_i-y_k)(r_i-r_k)-(y_j-y_i)(r_j-r_i)}{x_1-a_1}\vspace{3pt}\\
&=(y_i-y_k)\Big(\frac{y_i-y_k}{F(y_i)-F(y_k)}-\frac{y_j-y_k}{F(y_j)-F(y_k)}\Big)-(y_j-y_i)\Big(\frac{y_j-y_k}{F(y_j)-F(y_k)}-\frac{y_j-y_i}{F(y_j)-F(y_i)}\Big)\vspace{3pt}\\
&= \Big(\frac{(y_i-y_k)(F(y_j)-F(y_i))-(y_j-y_i)(F(y_i)-F(y_k))}{F(y_j)-F(y_k)}\Big)\Big(\frac{y_i-y_k}{F(y_i)-F(y_k)}-\frac{y_j-y_i}{F(y_j)-F(y_i)}\Big)>0.

\end{array}\]

For the second part, similarly to \ref{H1}, we have 
\[\hspace{-6pt}\begin{array}{ll}

&\frac{(c(z_j)-c(z_i))(r_j-r_i)-(c(z_i)-c(z_k))(r_i-r_k)}{x_1-a_1}\vspace{3pt}\\
=&\hspace{-8pt}(c(z_j)-c(z_i))\Big(\frac{y_j-y_k}{F(y_j)-F(y_k)}-\frac{y_j-y_i}{F(y_j)-F(y_i)}\Big)-(c(z_i)-c(z_k))\Big(\frac{y_i-y_k}{F(y_i)-F(y_k)}-\frac{y_j-y_k}{F(y_j)-F(y_k)}\Big)\vspace{3pt}\\
=&\hspace{-8pt}\frac{(F(y_j)-F(y_i))(y_i-y_k)-(F(y_i)-F(y_k))(y_j-y_i)}{F(y_j)-F(y_k)}\Big(\frac{c(z_j)-c(z_i)}{F(y_j)-F(y_i)}-\frac{c(z_i)-c(z_k)}{F(y_i)-F(y_k)}\Big)>0
\end{array}\]
from the condition \ref{condition2} on $F$ and $c.$ Hence, $a_2I_1+I_2>0$ which completes our proof.
\end{proof}

\begin{lem}\label{between3}
     Let $ u \in \mathcal{U} $, and assume that $ c $ and $ F $ satisfy condition \ref{condition2}. Suppose there exist indices $ k < j - 1 $ such that $ \mu(X_k) > 0 $ and $ \mu(X_j) > 0 $, while $ \mu(X_i) = 0 $ for all $ k < i < j $. Additionally, assume that $ X_k $ and $ X_j $ are adjacent, and the set of indifference points between $ X_k $ and $ X_j $  passes through the point $(1,0).$ Then, $u$ is not a solution of the monopolist's problem. 
\end{lem}

\begin{proof}

Let $k<i<j$ such that $\mu(X_i)=0.$ We perturb $u$ similarly to  Lemmas \ref{case1}, \ref{case2}, to get

\[\begin{array}{ll} &\mathcal{P}(u_a)-\mathcal{P}(u) \vspace{3pt}\\
      =& (-a_2(F(y_j)-F(y_i))-a_1(y_j-y_i)+(c(z_j)-c(z_i)))\mu( X_j\cap X_i^a)
      \vspace{3pt}\\
      &+(a_2(F(y_i)-F(y_k))+a_1(y_i-y_k)-(c(z_i)-c(z_k)))\mu( X_k\cap X_i^a)\vspace{3pt}\\
      
      =&(-a_2(F(y_j)-F(y_i))-(1-a_2\frac{F(y_j)-F(y_k)}{y_j-y_k})(y_j-y_i)+(c(z_j)-c(z_i)))\mu( X_j\cap X_i^a)\vspace{3pt}\\
      &+(a_2(F(y_i)-F(y_k))+(1-a_2\frac{F(y_j)-F(y_k)}{y_j-y_k})(y_i-y_k)-(c(z_i)-c(z_k)))\mu( X_k\cap X_i^a)\vspace{3pt}\\
      \end{array}\]
    where $a_1=1-a_2\frac{F(y_j)-F(y_k)}{y_j-y_k}$ as $a$ moves on the indifference line between $X_k$ and $X_j.$ For small enough $a_2,$ it is sufficient to study the sign of 
    \begin{equation}\label{exp}-(y_j-y_i-(c(z_j)-c(z_i)))\mu( X_j\cap X_i^a)+(y_i-y_k-(c(z_i)-c(z_k)))\mu( X_k\cap X_i^a)).\end{equation}
\begin{enumerate}

\item Assume $y_j-y_i-(c(z_j)-c(z_i))>0.$ We extend $f$ to $[0,2]^2$ such that $\alpha\le f\le \|f\|_\infty$ and let $X_{i,j}^a=(X_j\cap X_i^a)\cup B^a$ where $B^a=\{(x_1,x_2)\in [0,2]^2:\, x_1>1 \text{ and } x_2<-\frac{y_j-y_i}{F(y_j)-F(y_i)}(x_1-a_1)+a_2\}.$ Then,
\[\begin{array}{ll}
&-(y_j-y_i-(c(z_j)-c(z_i)))\mu( X_j\cap X_i^a)+(y_i-y_k-(c(z_i)-c(z_k)))\mu( X_k\cap X_i^a)\vspace{3pt}\\
\ge&-(y_j-y_i-(c(z_j)-c(z_i)))\mu( X_{j,i}^a)+(y_i-y_k-(c(z_i)-c(z_k)))\mu( X_k\cap X_i^a)
\end{array}\] which is similar to expression \eqref{c eq} and similarly we prove 
\[-(y_j-y_i-(c(z_j)-c(z_i)))\mu( X_{j,i}^a)+(y_i-y_k-(c(z_i)-c(z_k)))\mu( X_k\cap X_i^a)>0.\]

\item Assume that $y_j-y_i-(c(z_j)-c(z_i))\le0.$ If $y_i-y_k-(c(z_i)-c(z_k))>0,$ then expression \eqref{exp}  is positive. If $y_i-y_k-(c(z_i)-c(z_k))\le0,$ we extend $f$ to $[0,1]\times[-2,2]$ such that $\alpha\le f\le \|f\|_\infty$ and let $X_{i,k}^a=(X_k\cap X_i^a)\cup B_a$ where $B_a=\{(x_1,x_2)\in [0,1]\times[-2,2]:\, x_2<0 \text{ and } x_2>-\frac{y_i-y_k}{F(y_i)-F(y_k)}(x_1-a_1)+a_2\}.$
Then,
\[\begin{array}{ll}
&-(y_j-y_i-(c(z_j)-c(z_i)))\mu( X_j\cap X_i^a)+(y_i-y_k-(c(z_i)-c(z_k)))\mu( X_k\cap X_i^a))\vspace{3pt}\\
\ge&-(y_j-y_i-(c(z_j)-c(z_i)))\mu(X_j\cap X_i^a)+(y_i-y_k-(c(z_i)-c(z_k)))\mu( X_{i,k}^a)
\end{array}\] which is similar to expression \eqref{c eq2} and similarly we prove 
\[-(y_j-y_i-(c(z_j)-c(z_i)))\mu(X_j\cap X_i^a)+(y_i-y_k-(c(z_i)-c(z_k)))\mu( X_{i,k}^a)>0.\]
\end{enumerate}
This proves that expression \eqref{exp} is positive which implies $\mathcal{P}(u_a)-\mathcal{P}(u)>0$ for small enough $a_2,$ and hence $u$ is not a solution.
\end{proof}

Now we prove Theorem \ref{between}:
\begin{proof}
    Let $X_p$ and $X_s$ regions of $u$ such that they have positive masses. Suppose that  there exists $i$ where $p<i<s$ and $X_i$ has zero mass. Then, using Lemma \ref{partition} the assumptions in one of the above Lemmas \ref{between1}, \ref{between2}, \ref{between3} are satisfied, and so $u$ is not a solution, which proves the theorem. 
\end{proof}
We next state and prove a lemma about the structure arising from utility functions leading to intersecting indifference curves.
\begin{lem}\label{triangle}
   Let $u\in \mathcal{U}$ and suppose that two segments of indifference points intersect in $X.$ Then, there exists a region $X_i$ which shares boundary segments with only two adjacent regions and these segments intersect in $X.$ Moreover, at least one of the following is true:
   \begin{enumerate}
       \item The boundary segments intersect $[0,1]\times\{0\}$.
       \item  The boundary segments intersect $\{1\}\times[0,1]$. 
\item One of the boundary segments intersects $[0,1]\times\{0\}$ and the other intersects $\{1\}\times[0,1]$.

   \end{enumerate}
\end{lem}
\begin{proof} Consider the intersection $(\overline{x}_1,\overline{x}_2)$ with the smallest second component. We claim that there exists a region $X_i$ such that  $(\overline{x}_1,\overline{x}_2)\in\overline{X_i}$ and $\overline{x}_2=\max\{x_2:\, (x_1,x_2)\in \overline{X_i} \text{ for some } 0\le x_1\le1\}$.  We will prove that two of the segments that pass through $(\overline{x}_1,\overline{x}_2)$ have $(\overline{x}_1,\overline{x}_2)$ as their left end point.

    Suppose that two of the segments have $(\overline{x}_1,\overline{x}_2)$ as the right end point. If one of the segments has a positive slope, this would imply that there exists $z_p$ such that $z_i-z_p$ has negative slope which contradicts the fact that $F$ is increasing. Hence, both segments have negative slopes.
    
    From the convexity of the sets $(X_k),$ we conclude that there are two segments that are boundary segments of some set  $X_j$ with   $(\overline{x}_1,\overline{x}_2)$ as the point with the lowest second component in its closure. By a similar argument to the one in the proof of Lemma \ref{partition} we can find $\beta\in[0,1]$ such that $\frac{\partial u}{\partial x_1}(x_1,\beta)$ decreases in some neighborhood of $\overline{x}_1$ which contradicts the convexity of $u.$ This implies that we have up to one segment with $(\overline{x}_1,\overline{x}_2)$ as the right endpoint. From the convexity of sets $(X_s),$ there are at least three segments that pass through each intersection point and we conclude that there exist two segments with $(\overline{x}_1,\overline{x}_2)$ as the left endpoint.
    
    Then, two of the segments have $(\overline{x}_1,\overline{x}_2)$ as the left end point with negative slopes which implies the existence of $X_i$ with $(\overline{x}_1,\overline{x}_2)$ as the point with the greatest second component in its closure. $X_i$ has only two adjacent regions, because if not the boundary of $X_i$ would have at least three segments in $X$, and so we can find another intersection with lower second component than $(\overline{x}_1,\overline{x}_2).$ Hence, by Lemma \ref{partition}, both segments decrease until they intersect $([0,1]\times\{0\}) \cup(\{1\}\times[0,1])$ giving us one of the three cases stated above.
    \end{proof}
Our proof of Theorem \ref{thm: no intersection}, that indifference curves cannot intersect is split into several cases, based on where the line segments in the region described in the preceding lemma intersect the boundary.
\begin{lem}\label{case1}
         Suppose that $u \in \mathcal{U}$ and $X_i =(Du)^{-1}(z_i)$ shares a boundary with only two other regions, $X_{i-1}$ and $X_{i+1}$.  Suppose that the two indifference curves intersect within $\overline{X}$, and that both intersect $[0,1]\times\{0\}$.  Then, under hypothesis   \ref{condition3}, $u$ is not a solution to the monopolist's problem.   
    \end{lem}
\begin{proof}We will show that lowering the prices of good $y_i$ by $\epsilon$ strictly increases profits.  Let $v_j = \max_{x \in X}x \cdot z_j - u(x)$ be the pricing schedule corresponding to $u$ and change $v_i \rightarrow v_i -\epsilon$, while leaving the other prices $v_j, j \neq i$ unchanged.

Letting $X_i^\epsilon$ be the region of consumers who choose good $z_i$ under the new price schedule.  We then have
\[\begin{array}{ll}
    \mathcal P(v^\epsilon) = \mathcal P(v)&-\int_{X_i^\epsilon} \epsilon f(x) dx\\
    &+ \int_{X_{i}^\epsilon \cap X_{i+1}}[x \cdot (z_{i}-z_{i+1}) - (u_{i}+\epsilon-u_{i+1}) -c(z_{i}) +c(z_{i+1}) ]f(x)dx \\
    &+ \int_{X_{i}^\epsilon \cap X_{i-1}} [x \cdot (z_{i}-z_{i-1}) - (u_{i}+\epsilon-u_{i-1}) -c(z_{i}) +c(z_{i-1})]
    f(x) dx 
\end{array}\]
where $u_j=u(x)$ for all $x\in X_j.$
 We differentiate this expression with respect to $\epsilon$.  Note that as $\epsilon$ varies, the region $X_i^\epsilon$ expands outward along its boundary curves 
$$
 L_i^\epsilon=X^N_=(y_i, v_{i+1} +\epsilon -v_{i})\cap \overline{ X^\epsilon_i} =\{x: x\cdot (z_{i+1}-z_i) =v_{i+1} +\epsilon -v_{i}\} \cap \overline{ X^\epsilon_i} \subseteq \overline{  X_{i}^\epsilon} \cap \overline{X_{i+1}}
 $$ 
 and 
 $$
 L_{i-1}^\epsilon= X^N_=(y_{i-1}, v_{i} -\epsilon -v_{i-1})\cap \overline{ X^\epsilon_i} =\{x: x\cdot (z_{i}-z_{i-1}) =v_{i} -\epsilon -v_{i-1}\} \cap \overline{ X^\epsilon_i} \subseteq\overline{  X_{i}^\epsilon} \cap \overline{X_{i-1}}
 $$
 with outward unit normal speeds $\frac{1}{|z_i-z_{i+1}|}$ and $\frac{1}{|z_i-z_{i-1}|}$, respectively.  A standard formula from the calculus of moving boundaries then yields 
 \[\begin{array}{ll}
     &\frac{d}{d\epsilon} \mathcal P(v^\epsilon)  =\vspace{3pt}\\
     &-\int_{X_i^\epsilon}f(x)dx -\int_{L^\epsilon_i }\epsilon f(x)\frac{1}{|z_i-z_{i+1}|}d\mathcal H^{m-1}(x) 
     -\int_{L^\epsilon_{i-1}}\epsilon f(x)\frac{1}{|z_i-z_{i-1}|}d\mathcal H^{m-1}(x) \vspace{3pt}\\
     &-\int_{X_{i}^\epsilon \cap X_{i+1}}f(x)dx -\int_{X_{i}^\epsilon \cap X_{i-1}}f(x)dx\vspace{3pt}\\
     &+\int_{L^\epsilon_{i}}[x \cdot (z_{i}-z_{i+1}) - (u_{i}+\epsilon-u_{i+1}) -c(z_{i}) +c(z_{i+1}) ]f(x)\frac{1}{|z_i-z_{i+1}|}d\mathcal H^{m-1}(x)\vspace{3pt}\\
     &+
     \int_{L^\epsilon_{i-1}}[x \cdot (z_{i}-z_{i-1}) - (u_{i}+\epsilon-u_{i-1}) -c(z_{i}) -+(z_{i-1}) ]f(x)\frac{1}{|z_i-z_{i-1}|}d\mathcal H^{m-1}(x)
 \end{array}
 \]
Noting that $u_i=u_{i+1}$ and $u_i =u_{i-1}$ along the appropriate respective indifference curve, and that the volumes of the regions $X_{i}^\epsilon \cap X_{i+1}$ and $X_{i}^\epsilon \cap X_{i-1}$  dwindle to $0$ as $\epsilon \rightarrow 0$, we set $\epsilon =0$ to obtain
\begin{eqnarray}\label{optcond1}
         \frac{d}{d\epsilon}\Big|_{\epsilon=0} \mathcal P(v^\epsilon)  =
         &-&\int_{X_i}f(x)dx\label{derivative} \\
      &+&\int_{L^0_{i}}[x \cdot (z_{i}-z_{i+1})-c(z_{i}) +c(z_{i+1}) ]f(x)\frac{1}{|z_i-z_{i+1}|}d\mathcal H^{m-1}(x) \nonumber\\
&+&     \int_{L^0_{i-1}}[x \cdot (z_{i}-z_{i-1})-c(z_i) + c(z_{i-1}) ]f(x)\frac{1}{|z_i-z_{i-1}|}d\mathcal H^{m-1}(x) \nonumber
\end{eqnarray}
Now, let $a=(a_1,a_2)$ be the intersection point of the indifference regions $X^N_=(y_{i-1}, v_{i} - v_{i-1})$ and $X^N_=(y_i, v_{i+1} - v_{i})$.

Since both indifference curves reach axis $ [0,1]\times\{0\} $ by assumption, we can parametrize them  $(x^{i-1}_1(x_2),x_2)$ and $(x^i_1(x_2),x_2)$ by $x_ 2 \in [0,a_2]$ and since the line segment $X^N_=(y_i, v_{i+1} - v_{i})$ is orthogonal to $z_{i+1}-z_i =( y_{i+1},F(y_{i+1})) -(y_{i},F(y_{i})) $, the slope of $x_1^i$ is $\frac{F(y_{i+1})-F(y_i)}{y_{i+1} -y_i}$, and $1$-dimensional Hausdorff measure (ie, arclength) along it is given by $d\mathcal H^{m-1}(x) = \sqrt{\big(\frac{F(y_{i+1})-F(y_i)}{y_{i+1} -y_i}\big)^2+1}dx_2$, so that 
\[\hspace{-0.4cm}\begin{array}{ll}
&\int_{L_i^0}[x \cdot (z_{i}-z_{i+1})-c(z_{i}) +c(z_{i+1}) ]f(x)\frac{1}{|z_i-z_{i+1}|}d\mathcal H^{m-1}(x)\vspace{3pt}\\
&=[a \cdot (z_{i}-z_{i+1})-c(z_{i}) +c(z_{i+1}) ]\frac{1}{|z_i-z_{i+1}|}\sqrt{\big(\frac{F(y_{i+1})-F(y_i)}{y_{i+1} -y_i}\big)^2+1}\int_0^{a_2} f(x^i_1(x_2),x_2)dx_2\vspace{3pt}\\
&=[a \cdot (z_{i}-z_{i+1})-c(z_{i}) +c(z_{i+1}) ]\frac{1}{y_{i+1}-y_{i}}\int_0^{a_2} f(x^i_1(x_2),x_2)dx_2\vspace{3pt}\\
&= [-a_1 -a_2 \frac{F(y_{i+1})-F(y_i)}{y_{i+1}-y_{i}} +\frac{c(z_{i+1}) -c(z_{i})}{y_{i+1}-y_{i}}]\int_0^{a_2} f(x^i_1(x_2),x_2)dx_2,
\end{array}\]
where  $a\cdot(z_{i+1}-z_{i})=x\cdot(z_{i+1}-z_{i})$ along $L_i^0.$
Similarly,
\begin{eqnarray*}
\int_{L^0_{i-1}}[x \cdot (z_{i}-z_{i-1})-c(z_i) + c(z_{i-1}) ]f(x)\frac{1}{|z_i-z_{i-1}|}d\mathcal H^{m-1}(x)\\
=[a_1 +a_2 \frac{F(y_{i})-F(y_{i-1})}{y_{i}-y_{i-1}} -\frac{c(z_{i}) -c(z_{i-1})}{y_{i}-y_{i-1}}]\int_0^{a_2} f(x_1^{i-1}(x_2),x_2)dx_2.
\end{eqnarray*}

 We can therefore rewrite \eqref{derivative} as 
\begin{equation}\label{eqn: derivative computation}
\begin{array}{ll}
         &\frac{d}{d\epsilon}\Big|_{\epsilon=0} \mathcal P(v^\epsilon)  =\vspace{3pt}\\
         &-\int_{X_i}f(x)dx -a_1\Big[\int_0^{a_2} f(x^i_1(x_2),x_2)dx_2-\int_0^{a_2} f(x_1^{i-1}(x_2),x_2)dx_2\Big] \vspace{3pt}\\
&-a_2\Big[\frac{F(y_{i+1})-F(y_i)}{y_{i+1}-y_{i}} \int_0^{a_2} f(x^i_1(x_2),x_2)dx_2 - \frac{F(y_{i})-F(y_{i-1})}{y_{i}-y_{i-1}}\int_0^{a_2} f(x_1^{i-1}(x_2),x_2)dx_2\Big]      \vspace{3pt}\\
&+     \frac{c(z_{i+1}) -c(z_{i})}{y_{i+1}-y_{i}}\int_0^{a_2} f(x^i_1(x_2),x_2)dx_2 -\frac{c(z_{i}) -c(z_{i-1})}{y_{i}-y_{i-1}}\int_0^{a_2} f(x_1^{i-1}(x_2),x_2)dx_2. 
\end{array}
\end{equation}

We bound the first term of \eqref{eqn: derivative computation} below by 
\begin{equation}\label{X_i mass bound}
-\int_{X_i}f(x)dx\ge -\|f\|_\infty\frac{a_2^2}{2}\Big(\frac{F(y_{i+1})-F(y_i)}{y_{i+1}-y_i}-\frac{F(y_i)-F(y_{i-1})}{y_i-y_{i-1}}\Big).
\end{equation}

Turning to the second term in \eqref{eqn: derivative computation}, since $a_1\le1,$ we have

\[
        \begin{array}{ll}
        &-a_1\Big(\int_0^{a_2}f(x_1^{i}(x_2),x_2)-f(x_1^{i-1}(x_2),x_2)dx_2\Big)\vspace{3pt}\\
        &\ge - \|f_{x_1}\|_\infty\int_0^{a_2}\Big(\frac{F(y_{i+1})-F(y_i)}{y_{i+1}-y_i}-\frac{F(y_{i})-F(y_{i-1})}{y_{i}-y_{i-1}}\Big)(a_2-x_2)dx_2\vspace{3pt}\\
        &\ge-\|f_{x_1}\|_\infty\Big(\frac{F(y_{i+1})-F(y_i)}{y_{i+1}-y_i}-\frac{F(y_{i})-F(y_{i-1})}{y_{i}-y_{i-1}}\Big)\frac{a_2^2}{2}.
        \end{array}\]
For the third term in \eqref{eqn: derivative computation}, we have 
\[\begin{array}{ll}
          &-a_2\Big(\frac{F(y_{i+1})-F(y_i)}{y_{i+1}-y_i}\int_0^{a_2}f(x_1^{i}(x_2),x_2)dx_2-\frac{F(y_i)-F(y_{i-1})}{y_i-y_{i-1}}\int_0^{a_2}f(x_1^{i-1}(x_2),x_2)dx_2\Big)\vspace{3pt}\\
        =&-a_2\Bigg(\frac{F(y_{i+1})-F(y_i)}{y_{i+1}-y_i}\int_0^{a_2}(f(x_1^{i}(x_2),x_2)-f(x_1^{i-1}(x_2),x_2))dx_2\vspace{3pt}\\
        &+\Big(\frac{F(y_{i+1})-F(y_i)}{y_{i+1}-y_i}-\frac{F(y_i)-F(y_{i-1})}{y_i-y_{i-1}}\Big)\int_0^{a_2}f(x_1^{i-1}(x_2),x_2)dx_2\Bigg)\vspace{3pt}\\
        \ge&- \frac{F(y_{i+1})-F(y_i)}{y_{i+1}-y_i}\|f_{x_1}\|_\infty\Big(\frac{F(y_{i+1})-F(y_i)}{y_{i+1}-y_i}-\frac{F(y_{i})-F(y_{i-1})}{y_{i}-y_{i-1}}\Big)\frac{a_2^2}{2}\vspace{3pt}\\
        &-\|f\|_\infty a_2^2\Big(\frac{F(y_{i+1})-F(y_i)}{y_{i+1}-y_i}
        -\frac{F(y_i)-F(y_{i-1})}{y_i-y_{i-1}}\Big),
        \end{array}\]
and we use $a_2^3\le a_2^2$ as $a_2\le1.$

 For the last term in \eqref{eqn: derivative computation}, we have
\[\hspace{-6pt}\begin{array}{ll}
    &\hspace{-8pt}\frac{c(z_{i+1})-c(z_i)}{y_{i+1}-y_i}\int_0^{a_2}f(x_1^{i}(x_2),x_2)dx_2-\frac{c(z_{i})-c(z_{i-1})}{y_i-y_{i-1}}\int_0^{a_2}f(x_1^{i-1}(x_2),x_2)dx_2\vspace{3pt}\\
    =&\hspace{-8pt}\frac{c(z_{i+1})-c(z_i)}{y_{i+1}-y_i}\int_0^{a_2}(f(x_1^{i}(x_2),x_2)-f(x_1^{i-1}(x_2),x_2))dx_2\vspace{3pt}\\
    &\hspace{-8pt}+\Big(\frac{c(z_{i+1})-c(z_i)}{y_{i+1}-y_i}-\frac{c(z_{i})-c(z_{i-1})}{y_i-y_{i-1}}\Big)\int_0^{a_2}f(x_1^{i-1}(x_2),x_2)dx_2\vspace{3pt}\\
    \ge&\hspace{-8pt}-\|f_{x_1}\|_\infty\frac{c(z_{i+1})-c(z_i)}{y_{i+1}-y_i}\Big(\frac{F(y_{i+1})-F(y_i)}{y_{i+1}-y_i}-\frac{F(y_{i})-F(y_{i-1})}{y_{i}-y_{i-1}}\Big)\frac{a_2^2}{2}\hspace{-3pt}+\hspace{-3pt}\alpha a_2\Big(\frac{c(z_{i+1})-c(z_i)}{y_{i+1}-y_i}-\frac{c(z_{i})-c(z_{i-1})}{y_i-y_{i-1}}\Big)
\end{array}\]

Using these bounds on (\ref{derivative}), we get
\[\hspace{-0.3cm}\begin{array}{ll}
&\frac{d}{d\epsilon}\Big|_{\epsilon=0} \mathcal P(v^\epsilon) \ge\vspace{3pt}\\
&\hspace{-0.3cm} -a_2^2\Big(\frac{F(y_{i+1})-F(y_i)}{y_{i+1}-y_i}-\frac{F(y_i)-F(y_{i-1})}{y_i-y_{i-1}}\Big)\Big(\frac{3}{2}\|f\|_\infty+\frac{\|f_{x_1}\|_\infty}{2}\Big(1+\frac{F(y_{i+1})-F(y_i)}{y_{i+1}-y_i}+\frac{c(z_{i+1})-c(z_i)}{y_{i+1}-y_i}\Big)\Big)\vspace{3pt}\\
&\hspace{-0.3cm}+\alpha a_2\Big(\frac{c(z_{i+1})-c(z_i)}{y_{i+1}-y_i}-\frac{c(z_{i})-c(z_{i-1})}{y_i-y_{i-1}}\Big)>0 \vspace{3pt}\\
\end{array}\]
by the fact that $a_2^2\le a_2$ and condition \ref{condition3}. This means decreasing the price of the $i$th good leads to a strictly larger profit.  Therefore, the original pricing schedule cannot be optimal.\end{proof}
\begin{lem}\label{case2} Suppose that $u \in \mathcal{U}$ and $X_i =(Du)^{-1}(z_i)$ shares a boundary with only two other regions, $X_{i-1}$ and $X_{i+1}$.  Suppose that the two indifference curves intersect within $\overline{X}$, and that both intersect $\{1\} \times [0,1]$.  Then, under hypothesis  \ref{condition4}, $u$ is not a solution to the monopolist's problem.   
\end{lem}
\begin{proof}
    We perturb similarly to the proof of Lemma \ref{case1} to get 
    \begin{eqnarray}\label{optcond2}
         \frac{d}{d\epsilon}\Big|_{\epsilon=0} \mathcal P(v^\epsilon)  =
         &-&\int_{X_i}f(x)dx\label{derivative2} \\
      &+&\int_{L^0_{i}}[x \cdot (z_{i}-z_{i+1})-c(z_{i}) +c(z_{i+1}) ]f(x)\frac{1}{|z_i-z_{i+1}|}d\mathcal H^{m-1}(x) \nonumber\\
&+&     \int_{L^0_{i-1}}[x \cdot (z_{i}-z_{i-1})-c(z_i) + c(z_{i-1}) ]f(x)\frac{1}{|z_i-z_{i-1}|}d\mathcal H^{m-1}(x). \nonumber
\end{eqnarray}
    Since both indifference curves reach axis $\{1\} \times [0,1]$ by assumption, we can parametrize them  $(x_1,x^{i-1}_2(x_1))$ and $(x_1,x^i_2(x_1))$ by $x_ 1 \in [1-a_1,1]$ and since the line segment $X^N_=(y_i, v_{i+1} - v_{i})$ is orthogonal to $z_{i+1}-z_i =( y_{i+1},F(y_{i+1})) -(y_{i},F(y_{i})) $, the slope of $x_2^i$ is $-\frac{y_{i+1} -y_i}{F(y_{i+1})-F(y_i)}$, and  $d\mathcal H^{m-1}(x) = \sqrt{\big(\frac{y_{i+1} -y_i}{F(y_{i+1})-F(y_i)}\big)^2+1}dx_2$, so that  
    \begin{eqnarray*}
&&\int_{L_i^0}[x \cdot (z_{i}-z_{i+1})-c(z_{i}) +c(z_{i+1}) ]f(x)\frac{1}{|z_i-z_{i+1}|}d\mathcal H^{m-1}(x)\\
&=&[a \cdot (z_{i}-z_{i+1})-c(z_{i}) +c(z_{i+1}) ]\frac{1}{F(y_{i+1})-F(y_{i})}\int_{1-a_1}^{1} f(x_1,x^i_2(x_1))dx_2\\
&=& [-a_1\frac{y_{i+1} -y_i}{F(y_{i+1})-F(y_i)} -a_2  +\frac{c(z_{i+1}) -c(z_{i})}{F(y_{i+1})-F(y_{i})}]\int_{1-a_1}^{1} f(x_1,x^i_2(x_1))dx_2
\end{eqnarray*}

Similarly we get 
\begin{eqnarray*}
\int_{L^0_{i-1}}[x \cdot (z_{i}-z_{i-1})-c(z_i) + c(z_{i-1}) ]f(x)\frac{1}{|z_i-z_{i-1}|}d\mathcal H^{m-1}(x)\\
=[a_1\frac{y_{i} -y_{i-1}}{F(y_{i})-F(y_{i-1})} +a_2  -\frac{c(z_{i}) -c(z_{i-1})}{F(y_{i})-F(y_{i-1})}]\int_{1-a_1}^{1} f(x_1,x^{i-1}_2(x_1))dx_2.
\end{eqnarray*}
We rewrite equation (\ref{derivative2}) as follows
\begin{equation}\label{derivative3}
\begin{array}{ll}
         &\frac{d}{d\epsilon}\Big|_{\epsilon=0} \mathcal P(v^\epsilon)  =\vspace{3pt}\\
         &-\int_{X_i}f(x)dx \vspace{3pt}\\
      &+ (1-a_1)\Big(\frac{y_{i+1} -y_i}{F(y_{i+1})-F(y_i)}
\int_{1-a_1}^{1} f(x_1,x^i_2(x_1))dx_2- \frac{y_{i} -y_{i-1}}{F(y_{i})-F(y_{i-1})}
\int_{1-a_1}^{1} f(x_1,x^{i-1}_2(x_1))dx_2\Big) \vspace{3pt}\\
&-    a_2\Big(\int_{1-a_1}^{1} f(x_1,x^i_2(x_1))dx_2-
\int_{1-a_1}^{1} f(x_1,x^{i-1}_2(x_1))dx_2\Big) \vspace{3pt}\\
&-  \Big(\frac{y_{i+1} -y_i}{F(y_{i+1})-F(y_i)}
\int_{1-a_1}^{1} f(x_1,x^i_2(x_1))dx_2-\frac{y_{i} -y_{i-1}}{F(y_{i})-F(y_{i-1})}
\int_{1-a_1}^{1} f(x_1,x^{i-1}_2(x_1))dx_2\Big)\vspace{3pt}\\
&+ \Big(\frac{c(z_{i+1}) -c(z_{i})}{F(y_{i+1})-F(y_{i})}\int_{1-a_1}^{1} f(x_1,x^i_2(x_1))dx_2
-\frac{c(z_{i}) -c(z_{i-1})}{F(y_{i})-F(y_{i-1})}\int_{1-a_1}^{1} f(x_1,x^{i-1}_2(x_1))dx_2\Big)
\end{array}\end{equation}

For the first term of (\ref{derivative3}), we have 
 \begin{equation}\label{mass bound2}
 -\int_{X_i} f(x)dx\ge - \|f\|_\infty \frac{(1-a_1)^2}{2}\Big(\frac{y_i-y_{i-1}}{F(y_i)-F(y_{i-1})}-\frac{y_{i+1}-y_i}{F(y_{i+1})-F(y_i)}\Big).
 \end{equation}

For the second term, we have 
\[\begin{array}{ll}

&\hspace{-3pt}-(1-a_1)\Big(\frac{y_{i}-y_{i-1}}{F(y_{i})-F(y_{i-1})}\int_{a_1}^{1}f(x_1,x_2^{i-1}(x_1))dx_1-\frac{y_{i+1}-y_i}{F(y_{i+1})-F(y_{i})}\int_{a_1}^{1}f(x_1,x_2^{i}(x_1))dx_1\Big)\vspace{3pt}\\
=&\hspace{-3pt}-(1-a_1)\Big(\frac{y_{i}-y_{i-1}}{F(y_{i})-F(y_{i-1})}\int_{a_1}^{1}(f(x_1,x_2^{i-1}(x_1))-f(x_1,x_2^{i}(x_1)))dx_1\vspace{3pt}\\
&\hspace{-3pt}+\Big(\frac{y_{i}-y_{i-1}}{F(y_{i})-F(y_{i-1})}-\frac{y_{i+1}-y_i}{F(y_{i+1})-F(y_{i})}\Big)\int_{a_1}^{1}f(x_1,x_2^{i}(x_1))dx_1\Big)\vspace{3pt}\\

\ge&\hspace{-3pt}-\frac{y_{i}-y_{i-1}}{F(y_{i})-F(y_{i-1})}\frac{(1-a_1)^2}{2}\|f_{x_2}\|_\infty\Big(\frac{y_i-y_{i-1}}{F(y_i)-F(y_{i-1})}-\frac{y_{i+1}-y_i}{F(y_{i+1})-F(y_i)}\Big)\vspace{3pt}\\
&\hspace{-3pt}-\Big(\frac{y_{i}-y_{i-1}}{F(y_{i})-F(y_{i-1})}-\frac{y_{i+1}-y_i}{F(y_{i+1})-F(y_{i})}\Big)\|f\|_\infty(1-a_1)^2
\end{array}\]
and we use the fact that $1-a_1\le1.$

As $a_2\le1,$ we have 
\[\begin{array}{ll}
&-a_2\Big(\int_{a_1}^{1}f(x_1,x_2^{i}(x_1))dx_1-\int_{a_1}^{1}f(x_1,x_2^{i-1}(x_1))dx_1\Big)\ge \vspace{3pt}\\
&-\frac{(1-a_1)^2}{2}\|f_{x_2}\|_\infty\Big(\frac{y_i-y_{i-1}}{F(y_i)-F(y_{i-1})}-\frac{y_{i+1}-y_i}{F(y_{i+1})-F(y_i)}\Big).
\end{array}\]
 Next, we have 
 \[\begin{array}{ll}
     & -\frac{y_{i+1}-y_i}{F(y_{i+1})-F(y_{i})}\int_{a_1}^{1}f(x_1,x_2^{i}(x_1))dx_1+\frac{y_{i}-y_{i-1}}{F(y_{i})-F(y_{i-1})}\int_{a_1}^{1}f(x_1,x_2^{i-1}(x_1))dx_1\vspace{3pt} \\
    
     =&-\Big(\frac{y_{i+1}-y_i}{F(y_{i+1})-F(y_{i})}-\frac{y_{i}-y_{i-1}}{F(y_{i})-F(y_{i-1})}\Big)\int_{a_1}^{1}f(x_1,x_2^{i}(x_1))dx_1\vspace{3pt}\\
     &+\frac{y_{i}-y_{i-1}}{F(y_{i})-F(y_{i-1})}\int_{a_1}^{1}(f(x_1,x_2^{i-1}(x_1))-f(x_1,x_2^{i}(x_1)))dx_1\vspace{3pt} \\
     
     \ge&\hspace{-0.25cm}\alpha(1-a_1)\Big(\frac{y_{i}-y_{i-1}}{F(y_{i})-F(y_{i-1})}-\frac{y_{i+1}-y_i}{F(y_{i+1})-F(y_{i})}\Big)\vspace{3pt}\\
     &-\frac{y_{i}-y_{i-1}}{F(y_{i})-F(y_{i-1})}\frac{(1-a_1)^2}{2}\|f_{x_2}\|_\infty\Big(\frac{y_i-y_{i-1}}{F(y_i)-F(y_{i-1})}-\frac{y_{i+1}-y_i}{F(y_{i+1})-F(y_i)}\Big).
\end{array}\]

From the last term, we get 
    \[\begin{array}{ll}
    &\frac{c(z_{i+1})-c(z_i)}{F(y_{i+1})-F(y_i)}\int_{a_1}^{1}f(x_1,x_2^{i}(x_1))dx_1-\frac{c(z_{i})-c(z_{i-1})}{F(y_i)-F(y_{i-1})}\int_{a_1}^{1}f(x_1,x_2^{i-1}(x_1))dx_1\vspace{3pt}\\
   = &\Big(\frac{c(z_{i+1})-c(z_i)}{F(y_{i+1})-F(y_i)}\int_{a_1}^{1}(f(x_1,x_2^{i}(x_1))-f(x_1,x_2^{i-1}(x_1)))dx_1\vspace{3pt}\\
   &+\Big(\frac{c(z_{i+1})-c(z_i)}{F(y_{i+1})-F(y_i)}-\frac{c(z_{i})-c(z_{i-1})}{F(y_i)-F(y_{i-1})}\Big)\int_{a_1}^{1}f(x_1,x_2^{i-1}(x_1))dx_1\Big)\vspace{3pt}\\
   \ge&-\frac{c(z_{i+1})-c(z_i)}{F(y_{i+1})-F(y_i)}\frac{(1-a_1)^2}{2}\|f_{x_2}\|_\infty\Big(\frac{y_i-y_{i-1}}{F(y_i)-F(y_{i-1})}-\frac{y_{i+1}-y_i}{F(y_{i+1})-F(y_i)}\Big)\vspace{3pt}\\
   &+\alpha(1-a_1)\Big(\frac{c(z_{i+1})-c(z_i)}{F(y_{i+1})-F(y_i)}-\frac{c(z_{i})-c(z_{i-1})}{F(y_i)-F(y_{i-1})}\Big).
    \end{array}\]

Hence,

\[\hspace{-0.25cm}\begin{array}{ll}
\frac{d}{d\epsilon}\Big|_{\epsilon=0} \mathcal P(v^\epsilon)\ge&-\frac{(1-a_1)^2}{2}\Big(\frac{y_i-y_{i-1}}{F(y_i)-F(y_{i-1})}-\frac{y_{i+1}-y_i}{F(y_{i+1})-F(y_i)}\Big)\times\vspace{3pt}\\
&\Big(3\|f\|_\infty+\|f_{x_2}\|_\infty\Big(1+2\frac{y_{i}-y_{i-1}}{F(y_{i})-F(y_{i-1})}+\frac{c(z_{i+1})-c(z_i)}{F(y_{i+1})-F(y_i)}\Big)\Big)\vspace{3pt}\\
&+(1-a_1)\alpha\Big(\frac{y_{i}-y_{i-1}}{F(y_{i})-F(y_{i-1})}-\frac{y_{i+1}-y_i}{F(y_{i+1})-F(y_{i})}+\frac{c(z_{i+1})-c(z_i)}{F(y_{i+1})-F(y_i)}-\frac{c(z_{i})-c(z_{i-1})}{F(y_i)-F(y_{i-1})}\Big).
\end{array}\]
Now, consider the the triangle formed by $(1,0),\,(1,1)$ and $D=(\beta,1)$ where $D$ is the intersection between $[0,1]\times\{1\}$ and the line  of equation $x_2=-\frac{y_{i+1}-y_{i}}{F(y_{i+1})-F(y_{i})}(x_1-1)$ which is parallel to $L_{i}^0$ and passes through $(1,0).$ The lower angle of the triangle is equal to the angle $\theta_{i}$ between $z_{i+1}-z_{i}$ and the $x_1-$axis, which means that $\tan(\theta_{i})=\frac{F(y_{i+1})-F(y_{i})}{y_{i+1}-y_{i}}=\frac{1-\beta}{1}$ and then 
\begin{equation}\label{tan ineq}
\frac{F(y_{i+1})-F(y_{i})}{y_{i+1}-y_{i}}=1-\beta\ge 1-a_1.
\end{equation}

Hence,

\[\begin{array}{ll}\hspace{-0.2cm}
\frac{d}{d\epsilon}\Big|_{\epsilon=0} \mathcal P(v^\epsilon)\ge&-\frac{1}{2}\Big(\frac{F(y_{i+1})-F(y_{i})}{y_{i+1}-y_{i}}\Big)(1-a_1)\Big(\frac{y_i-y_{i-1}}{F(y_i)-F(y_{i-1})}-\frac{y_{i+1}-y_i}{F(y_{i+1})-F(y_i)}\Big)\times \vspace{3pt}\\
&\Big(3\|f\|_\infty+\|f_{x_2}\|_\infty\Big(1+2\frac{y_{i}-y_{i-1}}{F(y_{i})-F(y_{i-1})}+\frac{c(z_{i+1})-c(z_i)}{F(y_{i+1})-F(y_i)}\Big)\Big)\vspace{3pt}\\
&+(1-a_1)\alpha\Big(\frac{y_{i}-y_{i-1}}{F(y_{i})-F(y_{i-1})}-\frac{y_{i+1}-y_i}{F(y_{i+1})-F(y_{i})}+\frac{c(z_{i+1})-c(z_i)}{F(y_{i+1})-F(y_i)}-\frac{c(z_{i})-c(z_{i-1})}{F(y_i)-F(y_{i-1})}\Big)
\end{array}\]
 which is positive by condition \ref{condition4} and the fact that $3\|f\|_\infty\hspace{-3pt}\le\hspace{-0.1cm} \Big( 2+\frac{\frac{F(y_{i+1})-F(y_{i})}{y_{i+1}-y_{i}}}{\frac{F(y_i)-F(y_{i-1})}{y_i-y_{i-1}}}\Big)\hspace{-0.1cm}\|f\|_\infty$. Therefore, $u$ is not a solution.
\end{proof}

\begin{lem}\label{case3}
     Suppose that $u \in \mathcal{U}$ and $X_i =(Du)^{-1}(z_i)$ shares a boundary with only two other regions, $X_{i-1}$ and $X_{i+1}$.  Suppose that the two indifference curves intersect within $\overline{X}$, and one of them intersects $[0,1]\times\{0\}$ and the other intersects $\{1\} \times [0,1]$.  Then, under hypotheses \ref{condition3} and \ref{condition4}, $u$ is not a solution to the monopolist's problem.   
\end{lem}
\begin{proof}
 We perturb similarly to the proof of Lemma \ref{case1} to get 
    \[\begin{array}{ll}
         &\frac{d}{d\epsilon}\Big|_{\epsilon=0} \mathcal P(v^\epsilon)  =-\int_{X_i}f(x)dx \vspace{3pt}\\
      &+\int_{L^0_{i}}[-a_1(y_{i+1}-y_i)-a_2(F(y_{i+1})-F(y_i)) +c(z_{i+1})-c(z_{i}) ]f(x)\frac{1}{|z_i-z_{i+1}|}d\mathcal H^{m-1}(x) \vspace{3pt}\\
&+    \int_{L^0_{i-1}}[a_1(y_i-y_{i-1})+a_2(F(y_{i})-F(y_{i-1}))-c(z_i) + c(z_{i-1}) ]f(x)\frac{1}{|z_i-z_{i-1}|}d\mathcal H^{m-1}(x). 
\end{array}\]
\begin{enumerate}
    \item If $ a_1(y_{i}-y_{i-1})+a_2(F(y_{i})-F(y_{i-1})) -(c(z_{i})-c(z_{i-1}))\le0,$ then
    \[ \begin{array}{ll}
        & \frac{d}{d\epsilon}\Big|_{\epsilon=0} \mathcal P(v^\epsilon)  \ge\vspace{3pt}\\
         &-\int_{X_i}f(x)dx \vspace{3pt}\\
      &+\int_{L^0_{i}}[-a_1(y_{i+1}-y_i)-a_2(F(y_{i+1})-F(y_i)) +c(z_{i+1})-c(z_{i}) ]f(x)\frac{1}{|z_i-z_{i+1}|}d\mathcal H^{m-1}(x) \vspace{3pt}\\
&+     \int_{L'_{i-1}}[a_1(y_i-y_{i-1})+a_2(F(y_{i})-F(y_{i-1}))-c(z_i) + c(z_{i-1}) ]f(x)\frac{1}{|z_i-z_{i-1}|}d\mathcal H^{m-1}(x). 
\end{array}\]

where $L_{i-1}'$ is the segment connecting $a$ and the intersection between the line passing through $L_{i-1}^0$ and  $\{1\}\times[-\eta,\eta],$ and we extend $f$  to $\overline{f}\ge\alpha$ on $[0,1]\times[-\eta,\eta]$ where $\|\overline{f}\|_\infty=\|f\|_\infty$ and $\|\overline{f}_{x_2}\|_\infty=\|f_{x_2}\|_\infty,$ for large enough $\eta>0.$ And using the following inequality
     \[\mu( X_i)\le \|f\|_\infty\frac{(1-a_1)^2}{2}\Big(\frac{y_i-y_{i-1}}{F(y_i)-F(y_{i-1})}-\frac{y_{i+1}-y_i}{F(y_{i+1})-F(y_i)}\Big),\]

     we get that 
    \[\vspace{-4pt}\begin{array}{ll}
         &\frac{d}{d\epsilon}\Big|_{\epsilon=0} \mathcal P(v^\epsilon)\ge\vspace{3pt}\\  
         &
         - \|f\|_\infty\frac{(1-a_1)^2}{2}\Big(\frac{y_i-y_{i-1}}{F(y_i)-F(y_{i-1})}-\frac{y_{i+1}-y_i}{F(y_{i+1})-F(y_i)}\Big)\vspace{3pt} \\
      &+\int_{L^0_{i}}[-a_1(y_{i+1}-y_i)-a_2(F(y_{i+1})-F(y_i)) +c(z_{i+1})-c(z_{i}) ]f(x)\frac{1}{|z_i-z_{i+1}|}d\mathcal H^{m-1}(x) \vspace{3pt}\\
&+     \int_{L'_{i-1}}[a_1(y_i-y_{i-1})+a_2(F(y_{i})-F(y_{i-1}))-c(z_i) + c(z_{i-1}) ]f(x)\frac{1}{|z_i-z_{i-1}|}d\mathcal H^{m-1}(x). 
\end{array}\]
     which is similar to equation \eqref{optcond2} after using \eqref{mass bound2} in the proof of Lemma \ref{case2} and can be solved using the same argument to get $\frac{d}{d\epsilon}\Big|_{\epsilon=0} \mathcal P(v^\epsilon)>0.$

\item If $ a_1(y_{i}-y_{i-1})+a_2(F(y_{i})-F(y_{i-1})) -(c(z_{i})-c(z_{i-1}))\ge0,$ we have 2 cases.

\begin{enumerate}
\item If  $a_1(y_{i+1}-y_i)+a_2(F(y_{i+1})-F(y_i)) -(c(z_{i+1})-c(z_i))\ge0,$
 then 
\[\begin{array}{ll}
        & \frac{d}{d\epsilon}\Big|_{\epsilon=0} \mathcal P(v^\epsilon)   \ge\\&
         -\|f\|_\infty\Big(\frac{F(y_{i+1})-F(y_i)}{y_{i+1}-y_i}-\frac{F(y_{i})-F(y_{i-1})}{y_{i}-y_{i-1}}\Big)\frac{a_2^2}{2}\vspace{3pt}\\
      &+\int_{L'_{i}}[-a_1(y_{i+1}-y_i)-a_2(F(y_{i+1})-F(y_i)) +c(z_{i+1})-c(z_{i}) ]f(x)\frac{1}{|z_i-z_{i+1}|}d\mathcal H^{m-1}(x) \vspace{3pt}\\
&+    \int_{L^0_{i-1}}[a_1(y_i-y_{i-1})+a_2(F(y_{i})-F(y_{i-1}))-c(z_i) + c(z_{i-1}) ]f(x)\frac{1}{|z_i-z_{i-1}|}d\mathcal H^{m-1}(x). 
\end{array}\]
          where $L_{i}'$ is the segment connecting $a$ and the intersection between the line passing through $L^0_{i}$ and  $[0,2]\times\{0\},$ and we extend $f$ to $\overline{f}\ge\alpha$ on $[0,2]\times[-2,2]$ where $\|\overline{f}\|_\infty=\|f\|_\infty$ and $\|\overline{f}_{x_1}\|_\infty=\|f_{x_1}\|_\infty.$ 
          
          A similar argument to the one in the proof of Lemma \ref{case1} implies that $\frac{d}{d\epsilon}\Big|_{\epsilon=0} \mathcal P(v^\epsilon)>0.$

        \item  If $a_1(y_{i+1}-y_i)+a_2(F(y_{i+1})-F(y_i)) -(c(z_{i+1})-c(z_i))\le0:$

Let $(d,0)$ be the intersection of $L_{i-1}^0$ and the $x_1-$axis, and let $(1,e)$ be the intersection between $L_{i}^0$ and the line $x_1=1.$ Let $T$ be the area of the trapezoid formed by $(d,0),\,(1,e)\, (1,0)$ and $B$ where $B$ is the intersection between $L_{i}^0$ and the line $x_1=d.$ In addition, let $S$ be the area of the triangle formed by $a,\,B$ and $(d,0).$  We find an upper bound on the following ratio knowing that $d=a_1+\frac{F(y_i)-F(y_{i-1})}{y_i-y_{i-1}}a_2,$  $e=a_2-\frac{y_{i+1}-y_i}{F(y_{i+1})-F(y_i)}(1-a_1),$ and $B=(d,a_2-\frac{y_{i+1}-y_i}{F(y_{i+1})-F(y_i)}(d-a_1)),$
          \[\begin{array}{ll}
 
        \frac{T}{S}
        &=  \frac{(1-d)(e+a_2-\frac{y_{i+1}-y_i}{F(y_{i+1})-F(y_i)}(d-a_1))}{(d-a_1)(a_2-\frac{y_{i+1}-y_i}{F(y_{i+1})-F(y_i)}(d-a_1))}\vspace{3pt}\\
&=\frac{y_{i+1}-y_i}{F(y_{i+1})-F(y_i)}\frac{\Big(1-a_1-a_2\frac{F(y_i)-F(y_{i-1})}{y_i-y_{i-1}}\Big)\Big(a_2\frac{F(y_{i+1})-F(y_i)}{y_{i+1}-y_i}+a_1-1+a_2\frac{F(y_{i+1})-F(y_i)}{y_{i+1}-y_i}-a_2\frac{F(y_i)-F(y_{i-1})}{y_i-y_{i-1}}\Big)}{a_2^2\frac{F(y_i)-F(y_{i-1})}{y_i-y_{i-1}}\Big(1-\frac{\frac{F(y_i)-F(y_{i-1})}{y_i-y_{i-1}}}{\frac{F(y_{i+1})-F(y_i)}{y_{i+1}-y_i}}\Big)}\vspace{3pt}\\
&=\frac{\Big(1-a_1-a_2\frac{F(y_i)-F(y_{i-1})}{y_i-y_{i-1}}\Big)\Big(a_2\frac{F(y_{i+1})-F(y_{i})}{y_{i+1}-y_{i}}-a_2\frac{F(y_i)-F(y_{i-1})}{y_i-y_{i-1}}\Big)}{a_2^2\frac{F(y_i)-F(y_{i-1})}{y_i-y_{i-1}}\Big(\frac{F(y_{i+1})-F(y_{i})}{y_{i+1}-y_{i}}-\frac{F(y_i)-F(y_{i-1})}{y_i-y_{i-1}}\Big)}\vspace{3pt}\\
&\quad+\frac{\Big(1-a_1-a_2\frac{F(y_i)-F(y_{i-1})}{y_i-y_{i-1}}\Big)\Big(a_1-1+a_2\frac{F(y_{i+1})-F(y_{i})}{y_{i+1}-y_{i}}\Big)}{a_2^2\frac{F(y_i)-F(y_{i-1})}{y_i-y_{i-1}}\Big(\frac{F(y_{i+1})-F(y_{i})}{y_{i+1}-y_{i}}-\frac{F(y_i)-F(y_{i-1})}{y_i-y_{i-1}}\Big)}.\vspace{3pt}\\
          \end{array}\]

          As $(1-d)=1-a_1-a_2\frac{F(y_i)-F(y_{i-1})}{y_i-y_{i-1}}\le a_2\Big(\frac{F(y_{i+1})-F(y_{i})}{y_{i+1}-y_{i}}-\frac{F(y_i)-F(y_{i-1})}{y_i-y_{i-1}}\Big),$ we get
\[\begin{array}{ll}
\frac{T}{S}\le&\frac{\Big(1-a_1-a_2\frac{F(y_i)-F(y_{i-1})}{y_i-y_{i-1}}\Big)}{a_2\frac{F(y_i)-F(y_{i-1})}{y_i-y_{i-1}}}
+\frac{\Big(1-a_1-a_2\frac{F(y_i)-F(y_{i-1})}{y_i-y_{i-1}}\Big)\Big(a_1-1+a_2\frac{F(y_{i+1})-F(y_{i})}{y_{i+1}-y_{i}}\Big)}{a_2\frac{F(y_i)-F(y_{i-1})}{y_i-y_{i-1}}\Big(1-a_1-a_2\frac{F(y_i)-F(y_{i-1})}{y_i-y_{i-1}}\Big)},\vspace{3pt}\\
=&\frac{\frac{F(y_{i+1})-F(y_{i})}{y_{i+1}-y_{i}}-\frac{F(y_i)-F(y_{i-1})}{y_i-y_{i-1}}}{\frac{F(y_i)-F(y_{i-1})}{y_i-y_{i-1}}}.

          \end{array}\]
   Then 
   \[\hspace{-0.4cm}\begin{array}{ll}
   \mu( X_i)&\le\|f\|_\infty(T+S)\\
   &\le \|f\|_\infty\Big(\frac{\frac{F(y_{i+1})-F(y_{i})}{y_{i+1}-y_{i}}-\frac{F(y_i)-F(y_{i-1})}{y_i-y_{i-1}}}{\frac{F(y_i)-F(y_{i-1})}{y_i-y_{i-1}}}S+S\Big)\vspace{3pt}\\
   &=\|f\|_\infty\Big(\frac{\frac{F(y_{i+1})-F(y_{i})}{y_{i+1}-y_{i}}-\frac{F(y_i)-F(y_{i-1})}{y_i-y_{i-1}}}{\frac{F(y_i)-F(y_{i-1})}{y_i-y_{i-1}}}+1\Big)\frac{(d-a_1)^2}{2}\Big(\frac{y_i-y_{i-1}}{F(y_i)-F(y_{i-1})}-\frac{y_{i+1}-y_i}{F(y_{i+1})-F(y_i)}\Big)\vspace{3pt}\\
   &\le\|f\|_\infty\Big(\frac{\frac{F(y_{i+1})-F(y_{i})}{y_{i+1}-y_{i}}}{\frac{F(y_i)-F(y_{i-1})}{y_i-y_{i-1}}}\Big)\frac{(d-a_1)^2}{2}\Big(\frac{y_i-y_{i-1}}{F(y_i)-F(y_{i-1})}-\frac{y_{i+1}-y_i}{F(y_{i+1})-F(y_i)}\Big).\vspace{3pt}\\

   \end{array}\]

   Then, 
    \[\hspace{-0.8cm}\begin{array}{ll}
         &\frac{d}{d\epsilon}\Big|_{\epsilon=0} \mathcal P(v^\epsilon)  \ge\\
         &-\|f\|_\infty\Big(\frac{\frac{F(y_{i+1})-F(y_{i})}{y_{i+1}-y_{i}}}{\frac{F(y_i)-F(y_{i-1})}{y_i-y_{i-1}}}\Big)\frac{(d-a_1)^2}{2}\Big(\frac{y_i-y_{i-1}}{F(y_i)-F(y_{i-1})}-\frac{y_{i+1}-y_i}{F(y_{i+1})-F(y_i)}\Big)\vspace{3pt}\\
      &+\int_{L'_{i}}[-a_1(y_{i+1}-y_i)-a_2(F(y_{i+1})-F(y_i)) +c(z_{i+1})-c(z_{i}) ]f(x)\frac{1}{|z_i-z_{i+1}|}d\mathcal H^{m-1}(x) \vspace{3pt}\\
&+   \int_{L^0_{i-1}}[a_1(y_i-y_{i-1})+a_2(F(y_{i})-F(y_{i-1}))-c(z_i) + c(z_{i-1}) ]f(x)\frac{1}{|z_i-z_{i+1}|}d\mathcal H^{m-1}(x). 
\end{array}\]
   where $L_{i}'$ is the segment between $a$ and the point $B.$ By a similar argument to the one in the proof of Lemma \ref{case2}, we obtain  
\[
\begin{aligned}
&\frac{d}{d\epsilon} \Big|_{\epsilon=0} \mathcal{P}(v^\epsilon) \vspace{3pt}\\
\geq & -\frac{1}{2} \Big(\frac{F(y_{i+1}) - F(y_i)}{y_{i+1} - y_i} \Big)(d - a_1) 
\Big(\frac{y_i - y_{i-1}}{F(y_i) - F(y_{i-1})} - \frac{y_{i+1} - y_i}{F(y_{i+1}) - F(y_i)} \Big) \times \\[10pt]
& \Big( 3\|f\|_\infty + \|f_{x_2}\|_\infty \Big(1 + 2\frac{y_i - y_{i-1}}{F(y_i) - F(y_{i-1})} + \frac{c(z_{i+1}) - c(z_i)}{F(y_{i+1}) - F(y_i)} \Big) \Big)+ (d - a_1)\times \\[5pt]
&  \alpha \Big( \frac{y_i - y_{i-1}}{F(y_i) - F(y_{i-1})} - \frac{y_{i+1} - y_i}{F(y_{i+1}) - F(y_i)} 
+ \frac{c(z_{i+1}) - c(z_i)}{F(y_{i+1}) - F(y_i)} - \frac{c(z_i) - c(z_{i-1})}{F(y_i) - F(y_{i-1})} \Big) > 0,
\end{aligned}
\]
where we use the inequality  
\[
d - a_1 < \frac{F(y_{i+1}) - F(y_i)}{y_{i+1} - y_i},
\]  
which follows from an argument similar to the one used to establish \eqref{tan ineq}. This inequality ensures the positivity of $\frac{d}{d\epsilon} \Big|_{\epsilon=0} \mathcal{P}(v^\epsilon),$ as guaranteed by \ref{condition4}.
   \end{enumerate}
\end{enumerate}\end{proof}

We are now prepared to prove Theorem \ref{disc}.  With the lemmas from the previous section in hand, the proof is relatively straightforward.

\begin{proof}
Suppose that $u$ is a solution; we first verify that no two indifference segments intersect in $X$.  If two indifference curves \emph{do} intersect, Lemma \ref{triangle} yields a region $X_i$ sharing a boundary with only two adjacent regions, both of which intersect the lower right part of the boundary, and by Theorem \ref{between} and Lemma \ref{partition}, it must be adjacent to $X_{i-1}$ and $X_{i+1}$.  Lemmas \ref{case1}, \ref{case2} and \ref{case3} then yield a contradiction, and so we conclude that, indeed, we cannot have two indifference segments intersecting in $X$.  

It remains to show that this no intersection property implies nestedness.  We show this now. For this proof alone, closure is taken in $X$, rather than in $\mathbb{R}^m.$

    Let $i_0=\min\{p:\,\mu(X_p)>0\}.$ We have $X_{i_0}$ has only one adjacent region which is $X_{i_0+1}.$  Hence, $X_=^N(y_{i_0},k_{i_0})$ is the indifference segment between $X_{i_0}$ and $X_{i_0+1}$  for some $k_{i_0}.$ By construction, we have $\nu=Du_{\#}\mu,$ meaning  $\nu(y_{i_0})=\mu(X_{i_0}).$ As $X_{i_0}$ is adjacent to $X_{i_0+1},$ there exists $\beta\in[0,1]$ such that $\frac{\partial}{\partial x_1}u(x_1,\beta)$ increases from $y_{i_0}$ to $y_{i_0+1}$ which means $X_{i_0}=X^N_\le(y_{i_0},k_{i_0})$ and hence $k_{i_0}=k^N(y_{i_0}).$ Also, for some $k_{i_0+1}$ the set $X^N_=(y_{i_0+1},k_{i_0+1})$ is the boundary between $X_{i_0+1}$ and $X_{i_0+2}$ and it does not intersect with $X^N_=(y_{i_0},k^N(y_{i_0}))$ in $X.$ Then, $\overline{X_{i_0}\cup X_{i_0+1}}=X^N_\le(y_{i_0+1},k_{i_0+1})$ and $\mu(X_{i_0}\cup X_{i_0+1})=\mu(X^N_\le(y_{i_0+1},k_{i_0+1}))$  and by construction we get 
    \[\nu(\{y_{i_0},y_{i_0+1}\})=\mu(X_{i_0})+\mu(X_{i_0+1})=\mu(X^N_\le(y_{i_0+1},k_{i_0+1}))\] and we get $k_{i_0+1}=k^N(y_{i_0+1}).$

    Proceeding inductively, we get $\overline{\cup_{k=0}^iX_k}= \overline{X^N_\le(y_{i-1},k^N(y_{i-1}))\cup X_{i}}=X^N_\le(y_i,k_i)$ where $X^N_=(y_i,k_i)$ is the boundary between $X_i$ and $X_{i+1}.$ Then 
\[\nu_N(\{y_p:\,0\le p\le i\})=\sum_{k=0}^{i}\mu(X_k)=\mu(X^N_\le(y_{i-1},k^N(y_{i-1})))+\mu(X_{i})=\mu(X^N_\le(y_{i},k_i))\]
and so $k_i=k^N(y_i).$

 Therefore, \[X^N_\le(y_{i},k^N(y_{i}))\subset \overline{\cup_{k=0}^{j-1}X_k}\cup X_j\subseteq X_<(y_{j},k^N(y_{j}))\] for all $i<j$ such that $\mu(X_i),\mu(X_j)>0,$ which implies discrete nestedness of the optimal transport problem between $\mu$ and $\nu$. By Theorem \ref{between} and Proposition \ref{OT to PA discrete}, we get the discrete nestedness of the solution $u$ of \eqref{PA}. By Theorem \ref{nestchar}  we get $y^*(x)=y_i$ for all $x\in X_i.$
\end{proof}

The proof of Lemma \ref{upper_segment} follows.
\begin{proof}
    
Assume there exists $i$ such that $\overline{t_i}<1$ and $x(\overline{t_i})=(0,\overline{t_i}).$ Hence,

\[\frac{\partial}{\partial t_i}(\mathcal{P})_{|t_i=\overline{t_i}}=0,\] which implies that 
\[\hspace{-0.5cm}\begin{array}{ll}

&\quad (F(y_{i+1})-F(y_i))\sum_{k=i+1}^N\mu( X_k)+\frac{\partial}{\partial t_i}(\mu( X_i))(\sum_{k=0}^{i-1}(x(t_{k})\cdot (z_{k+1}-z_{k}))-c(z_i))\vspace{3pt}\\
&\quad +\frac{\partial}{\partial t_i}(\mu( X_{i+1}))(\sum_{k=0}^{i}(x(t_{k})\cdot (z_{k+1}-z_{k}))-c(z_{i+1}))\vspace{3pt}\\
&=(F(y_{i+1})-F(y_i))\sum_{k=i+1}^N\mu( X_k)+\frac{\partial}{\partial t_i}(\mu( X_i))(c(z_{i+1})-c(z_i)-t_i(F(y_{i+1})-F(y_i)))\\
&=0
\end{array}\]
where we use the fact that $\frac{\partial}{\partial t_i}(\mu( X_{i+1}))=-\frac{\partial}{\partial t_i}(\mu( X_{i}))<0.$
Hence,
\[\begin{array}{ll}\overline{t_i}=\frac{(F(y_{i+1})-F(y_i))\frac{\sum_{k=i+1}^N\mu( X_k)}{\frac{\partial}{\partial t_i}(\mu( X_i))}+c(z_{i+1})-c(z_i)}{F(y_{i+1})-F(y_i)}\ge\frac{c(z_{i+1})-c(z_i)}{F(y_{i+1})-F(y_i)}\ge\frac{c(z_1)-c(z_0)}{F(y_1)-F(y_0)}=\frac{c(z_1)}{F(y_1)}>1
\end{array}\] which is a contradiction.
\end{proof}

We present next the proof of Lemma \ref{m}.

\begin{proof} Note that as the price of each good $i>1$ clearly satisfies $v_i \geq c_i>0$ (where $c_i=c(z_i)$), then for $x$ sufficiently close to $0$ we have $\max_{1\le i\le N}\{x\cdot z_i-v_i\}<0 =x\cdot y_0-c(y_0)$; therefore, as positive mass of consumers choose the opt-out good, $\mu(X_0) >0$.

       For $i \geq 1$, suppose the inequality  $b((1,1),z_{i}) -b((1,1),z_{i-1}) -c_{i}+c_{i-1} >0$ holds for some $i$, but it is not bought.  Since the set of purchased goods is known to include $y_0$ and to be consecutive by Lemma \ref{between}, without loss of generality, assume that $i-1$ is the last bought good.  Nestedness of the solution implies that types $x$ near $(1,1)$ buy good $i-1$.  Lower the prices of the $i$th good to: 
$$
v_i =v_{i-1} +c_i-c_{i-1} +\epsilon.
$$
Then profits from good $i$ are higher than that from good $i-1$, so if consumers can be enticed to purchase it instead, profits will go up.  Note that for the highest consumer, we have
$$
 b((1,1),z_{i}) -v_i = b((1,1),z_{i}) -v_{i-1} -c_i+c_{i-1} -\epsilon
 >
 b((1,1),z_{i-1}) -v_{i-1}
$$
for small enough $\epsilon$.  Therefore, agents near $(1,1)$ will buy good $i$, increasing the profits and contradicting optimality of the previous pricing plan.  If this choice of $v_i$ induces an indifference curve with another good, we simply choose a higher $v_i$, so that the curve
$$
 b(x,z_{i}) -v_i = b(x,z_{i}) -v_{i-1}
$$
lies entirely in the region of consumers who originally purchased $i-1$.

Now, on the other hand, suppose that $i$ is the highest good that a consumer buys.  The nested structure implies that consumer $x=(1,1)$ buys it.  This means that the indifference curve between $i-1$ and $i$,
$$
 b(x,z_{i}) -v_i = b(x,z_{i-1}) -v_{i-1}
$$
passes below $(1,1)$.  Now assume the inequality \eqref{eqn: preference at cost} fails, so that
$$
  b((1,1),z_{i}) -b((1,1),z_{i-1}) -c_{i}+c_{i-1} \leq 0
$$
Now note that for $x <(1,1)$ component wise along the indifference curve, we have
$$
v_i-c_i =b(x,z_i) -b(x,z_{i-1})+ v_{i-1}-c_i < b((1,1), z_{i}) -b((1,1), z_{i-1})+ v_{i-1} -c_i \leq v_{i-1} -c_{i-1}.
$$
Therefore, profits from $i-1$ are higher than those from $i$.  Raising prices slightly for good $i$ then increases profits from those buying good $i$ while also pushing some to switch to good $i-1$, without altering the rest of the solution. This contradicts optimality of the original plan.
\end{proof}

The proof of Theorem \ref{cont.nestedness} is broken into several lemmas. 
\begin{lem}\label{conv}
    Let $(u_N)$ be a sequence of solutions of the monopolist's problem \eqref{PA} with data $(\mu, Y_N, c).$ Then, there exists a subsequence $(u_{N_k})$ such that $u_{N_k}\to u$ uniformly as $k\to\infty,$ where $u$ is a solution of the monopolist's problem \eqref{PA} with data $(\mu,Y,c)$.
\end{lem}

\begin{lem}\label{noatoms}
 Suppose $\nu_N$ converges weakly to $\nu,$ and for each $N,$ the support of $\nu_N$ is consecutive; that is, $\{i:\nu_N(y_i)>0\}=\{i:0\le q_N\le i\le r_N\le N\}$ for some integers $q_N$ and $r_N.$  Let $\underline{y},\overline{y}\in Y$ such that $\nu(\{\underline{y}\})=\nu(\{\overline{y}\})=0$. If the optimal transport problem with marginals $(\mu,\nu_N)$ is discretely nested for all $N,$ then $X_=(\underline{y},k(\underline{y}))$ and $X_=(\overline{y},k(\overline{y}))$ do not intersect in  $X$.
\end{lem}


Together with Lemma \ref{conv} and Proposition \ref{OT to PA}, the next result will easily imply Theorem \ref{cont.nestedness}.
\begin{prop}\label{con.nest}
    Under the assumptions of Lemma \ref{noatoms}, let $(u_N)$ be a sequence of functions such that $u_N\to u$ uniformly, for some function $u,$ where $u_N$ is the solution of the dual problem of $(\mu,\nu_N),$ such that $\nu_N\to \nu$ weakly, for some $\nu.$ Then u solves the dual problem \eqref{eqn: OT dual} of $(\mu,\nu)$ and, if $(\mu,\nu_N)$ is discretely nested for all $N,$ then $(\mu,\nu)$ is nested. Moreover, the support of $\nu$ is connected.
\end{prop}

Next is the proof of Lemma \ref{conv}.
\begin{proof}
As the sequence of measures $(\nu_N) :=((Du_N)_\#\mu)$ corresponding to the solutions $u_N$ of the monopolist's problem with data $(\mu, Y_N,c),$ all have support within the compact set $Y$, there exists a weak-convergent subsequence $\nu_N\to\nu$ for some probability measure $\nu$ on $Y$.
From the stability of the optimal transport problem  $(\mu,\nu_N)$ with surplus $b(x,y)=x\cdot z(y)$ \cite{Santambrogio}, we get that the corresponding payoff $u_N$ and pricing functions $v_N$ converges uniformly to $u$ and $v$ respectively, the corresponding payoff and pricing functions of the optimal transport problem with marginals $(\mu, \nu)$ (that is, solution to the dual problem \eqref{eqn: OT dual}). Also, we conclude that
\[\mathcal{W}_b(\mu,\nu_N):=\int_X b(x,Du_N(x))d\mu(x)\to \mathcal{W}_b(\mu,\nu)=\int_X b(x,Du(x))d\mu(x).\]

And from the uniform convergence of $(u_N)$ we get that
\[\int_X u_Nd\mu\to\int_X u d\mu.\]

Also, from the weak convergence of $\nu_N,$ we deduce that
\[\int_X c(Du_N)d\mu=\int_Yc(y)d\nu_N\to \int_Y c(y)d\nu=\int_X c(Du)d\mu.\]
Hence,
\[\mathcal{P}(u_N)\to\mathcal{P}(u).\]

Let $\overline{\nu}=D\overline u_\#\mu$ be the corresponding measure of the solution $\overline{u}$ of the monopolist's problem \eqref{PA} with data $(\mu,Y,c).$ There exists a sequence of discrete measures $\overline{\nu}_N$ such that for all $N$ the atoms belongs to $Y_N,$ and $\overline{\nu}_N\to\overline{\nu}.$ Similar to the previous argument we get that $\mathcal{P}(\overline{u}_N)\to\mathcal{P}(\overline{u}),$ where $\overline{u}_N$ are the corresponding payoffs for the optimal transport problem $(\mu,\overline{\nu}_N).$ Hence, \[\mathcal{P}(\overline{u})=\lim_{N\to\infty}\mathcal{P}(\overline{u}_N)\le\lim_{N\to\infty}\mathcal{P}(u_N)=\mathcal{P}(u)\le\mathcal{P}(\overline{u}).\]
Therefore, $u$ is a solution for the monopolist's problem \eqref{PA} with data $(\mu,Y,c).$
\end{proof}
We now prove Lemma \ref{noatoms}.
\begin{proof}
    Let $\underline{y},\overline{y}\in Y$ such that $\underline{y}<
\overline{y},$ and $\underline{y}$ and $\overline{y}$ are not  atoms with respect to $\nu.$ 
Hence,\[\begin{array}{ll}\limsup_{N\to\infty}\nu_N([0,\underline{y}])\le\nu([0,\underline{y}])&=\nu([0,\underline{y}))\\
&\le\liminf_{N\to \infty}\nu_N([0,\underline{y}))\\
&\le \limsup_{N\to \infty}\nu_N([0,\underline{y}))\le\limsup_{N\to\infty}\nu_N([0,\underline{y}]),
\end{array}\]
    where the first and second inequalities comes from the definition of weak-convergence as $[0,\underline{y}]$ and $[0,\underline{y})$ are relatively closed and open in $Y$ respectively, and the equality comes from the fact that $\underline{y}$ is not an atom. Similarly for $\overline{y}$ we get 
\[\lim_{N\to\infty}\nu_N([0,\overline{y}])=\nu([0,\overline{y}])=\nu([0,\overline{y})).\]
    Let  $y_i^N\in Y_N\,$ and $y_j^N\in Y_N$  where for each $N,\,  i=\min\{k\,:\,y_k^N\in[\underline{y},\overline{y}) \}$ and $\,j=\max\{k\,:\,y_k^N\in[\underline{y},\overline{y}) \}$ and $y^N_p<y^N_{p+1}.$     

Note that $y_i^N\to \underline{y},$  $y_j^N \to \overline{y},$  $\frac{F(y_{i}^N)-F(y_{i-1}^N)}{y_{i}^N-y_{i-1}^N}\to F'(\underline{y})$ and $\frac{F(y_{j+1}^N)-F(y_j^N)}{y_{j+1}^N-y_j^N}\to F'(\overline{y})$. 

\hspace{-6pt}Consider the upper points of intersection 
\[d_N\in \overline{X_=^N(y_{i-1}^N,k^N(y_{i-1}^N))}\cap \big(\{0\} \times [0,1] \cup[0,1] \times \{1\}\big).\] There exists a convergent subsequence of $(d_N)$ that converges to $d\in\partial X.$ We consider the set $D=X_<(\underline{y},d\cdot(1,F'(\underline{y})))=\Big\{x\in X:\, (x-d)\cdot(1,F'(\underline{y}))<0\Big\}.$ We  claim that $\mu(D)=\nu([0,\underline{y}]).$ 

Suppose that $\mu(D)<\nu([0,\underline{y}])=\mu(X_\le(\underline{y},k(\underline{y}))),$ then $X_=(\underline{y},d\cdot(1,F'(\underline{y})))\subset X_<(\underline{y},k(\underline{y})).$ Let $d_{\underline{y}}$ be the upper intersection in $\overline{X_=(\underline{y},k(\underline{y}))}\cap \partial X,$ so $d\neq d_{\underline{y}}.$ We claim that there exists $\varepsilon>0$ such that $X^N_\le(y_{i-1}^N,k^N(y_{i-1}^N))\subset X_<(\underline{y},k(\underline{y})-\varepsilon)$ for all $N$ large enough. As $d\in \overline{X_<(\underline{y},k(\underline{y}))}\setminus X_=(\underline{y},k(\underline{y}))$ we get $(d_{\underline{y}}-d)\cdot(1,F'(\underline{y}))>0$ and we  take $0<\varepsilon<(d_{\underline{y}}-d)\cdot(1,F'(\underline{y})),$ then for $x\in X^N_\le(y_{i-1}^N,k^N(y_{i-1}^N)),$ we have 
\[(x-d_N)\cdot\Big(1,\frac{F(y_{i}^N)-F(y_{i-1}^N)}{y_{i}^N-y_{i-1}^N}\Big)\le0.\]
Knowing that $k(\underline{y})=d_{\underline{y}}\cdot(1,F'(\underline{y}))$ we get
\[\begin{array}{ll}\hspace{-0.2cm}
&x\cdot(1,F'(\underline{y}))-k(\underline{y})+\varepsilon\vspace{3pt}\\
&=(x-d_{\underline{y}})\cdot (1,F'(\underline{y}))+\varepsilon \vspace{3pt}\\
&=(x-d_N)\cdot(1,F'(\underline{y}))+(d_N-d_{\underline{y}})\cdot(1,F'(\underline{y}))+\varepsilon\vspace{3pt}\\
&=(x-d_N)\cdot\Big(1,\frac{F(y_{i}^N)-F(y_{i-1}^N)}{y_{i}^N-y_{i-1}^N}\Big)+(x-d_N)\cdot\Big((1,F'(\underline{y}))
-\Big(1,\frac{F(y_{i}^N)-F(y_{i-1}^N)}{y_{i}^N-y_{i-1}^N}\Big)\Big)\vspace{3pt}\\
&\quad +(d_N-d)\cdot(1,F'(\underline{y}))+(d-d_{\underline{y}})\cdot(1,F'(\underline{y}))+\varepsilon\vspace{3pt}\\
&\le (x-d_N)\cdot\Big(1,\frac{F(y_{i}^N)-F(y_{i-1}^N)}{y_{i}^N-y_{i-1}^N}\Big)+(1,1)\cdot\Big((1,F'(\underline{y}))
-\Big(1,\frac{F(y_{i}^N)-F(y_{i-1}^N)}{y_{i}^N-y_{i-1}^N}\Big)\Big)\vspace{3pt}\\
&\quad+(d_N-d)\cdot(1,F'(\underline{y}))+(d-d_{\underline{y}})\cdot(1,F'(\underline{y}))+\varepsilon<0
\end{array}\]

where the second and third terms go to zero and the first is non-positive and the fourth term $(d-d_{\underline{y}})\cdot(1,F'(\underline{y}))+\varepsilon$ is negative by the choice of $\varepsilon.$ Therefore, for large enough $N,$ 
we have $x\cdot(1,F'(\underline{y}))-k(\underline{y})+\varepsilon<0$ and then $x\in X_<(\underline{y},k(\underline{y})-\varepsilon) $ which proves our claim. As $f\ge\alpha>0,$ we have 
\[\begin{array}{ll}
\mu(X_\le(\underline{y},k(\underline{y})))-\mu(X^N_\le(y_{i-1}^N,k^N(y_{i-1}^N)))&=\mu(X_\le(\underline{y},k(\underline{y}))\setminus X^N_\le(y_{i-1}^N,k^N(y_{i-1}^N)))\\
&\ge\mu(X_\le(\underline{y},k(\underline{y}))\setminus X_<(\underline{y},k(\underline{y})-\varepsilon))>0\end{array}\] as $X^N_\le(y_{i-1}^N,k^N(y_{i-1}^N))\subset X_<(\underline{y},k(\underline{y})-\varepsilon)\subset X_\le(\underline{y},k(\underline{y})).$ Hence, \[\lim_{N\to\infty}\mu(X_\le(\underline{y},k(\underline{y})))-\mu(X^N_\le(y_{i-1}^N,k^N(y_{i-1}^N)))>0.\] But, 
\[\lim_{N\to\infty}\mu(X_\le(\underline{y},k(\underline{y})))-\mu(X^N_\le(y_{i-1}^N,k^N(y_{i-1}^N)))=\lim_{N\to\infty}\nu([0,\underline{y}])-\nu_N([0,\underline{y}])=0\] which is a contradiction. Using a similar argument we can prove that $\mu(D)$ cannot be bigger than $\nu([0,\underline{y}]),$ which implies $\mu(D)=\nu([0,\underline{y}]).$  This establishes the claim.

Since $\mu(D)=\lim_{N\to\infty}\nu_N([0,{\underline{y}}])=\nu([0,{\underline{y}}]),$ we get $D=X_<({\underline{y}},k({\underline{y}}))$ and then $\partial D\cap X=X_=(\underline{y},k(\underline{y})).$ Similarly, we can prove that  $E=X_<(\overline{y},k(\overline{y}))$ where $E=\{x\in X:\, (x-e)\cdot(1,F'(\overline{y}))<0\}$ such that $e$ is the limit of a subsequence $(e_N)$ and $e_N\in \overline{X_=^N(y_j^N,k^N(y_j^N))}\cap(\{0\}\times [0,1]\cup [0,1]\times\{1\}).$

Since $(\mu,\nu_N)$ is discretely nested, each point $ x_N \in \mathbb{R}^2 \setminus X $ is the unique intersection of the lines $ X_=^N(y_i^N, k^N(y_i^N)) $ and $ X_=^N(y_j^N, k^N(y_j^N)) $, and hence satisfies the linear system
\begin{equation}\label{eq3}
(x_N - d_N) \cdot (z_i^N - z_{i-1}^N) = 0, \quad (x_N - e_N) \cdot (z_{j+1}^N - z_j^N) = 0
\end{equation}
where $z_r^N=(y_r^N,F(y_r^N))$. This system determines $ x_N $ uniquely since the direction vectors $ z_i^N - z_{i-1}^N $ and $ z_{j+1}^N - z_j^N $ are linearly independent for all large $ N $. As $ N \to \infty $, the data $ d_N, e_N $, and the direction vectors converge to limits $ d, e $ and $ (1, F'(\underline{y})), (1, F'(\overline{y})) $ respectively, (after multiplying equations \eqref{eq3}  by $\frac{1}{y_{i}^{N}-y_{i-1}^{N}}$ and $\frac{1}{y_{j+1}^{N}-y_{j}^{N}}$ respectively) , which are also linearly independent. Hence, the linear systems converge to a limiting system that remains invertible, and it follows that $ x_N \to \underline{x} $, the unique  solution to
\[
(\underline{x} - d) \cdot (1, F'(\underline{y})) = 0, \quad (\underline{x} - e) \cdot (1, F'(\overline{y})) = 0.
\]
Since $ \mathbb{R}^2 \setminus X $ is closed and each $ x_N \in \mathbb{R}^2 \setminus X $, we conclude that  $ \underline{x} \in \mathbb{R}^2 \setminus X $. Therefore,  $X_=(\underline{y},k(\underline{y}))$ and $X_=(\overline{y},k(\overline{y}))$ do not intersect in $X.$

\end{proof}
We turn now to the proof of Proposition \ref{con.nest}.
\begin{proof}
    Let $\underline{y},\overline{y}\in Y$ such that $\underline{y}<\overline{y}$. We will prove that $X_=(\underline{y},k_+(\underline{y}))$ does not intersect $X_=(\overline{y},k_-(\overline{y})).$   Suppose that $X_=(\underline{y},k_+(\underline{y}))$ intersects $X_=(\overline{y},k_-(\overline{y})).$ For all $\delta_0>0,$ there exists $\delta_0>\delta>0$ such that $\overline{y}_\delta$ is not an atom.  We claim that for small enough $\delta,$  $X_=(\overline{y}_\delta,k(\overline{y}_\delta))$ intersects $X_=(\underline{y},k_+(\underline{y})).$ Suppose that for all $\delta>0,$ $X_=(\overline{y}_\delta,k(\overline{y}_\delta))$ does not intersect $X_=(\underline{y},k_+(\underline{y})).$ Let $x_0\in X_<(\underline{y},k_+(\underline{y}))\setminus X_\le(\overline{y},k_-(\overline{y})),$ which means 
    \begin{equation}\label{delta}
    x_0\cdot(1,F'(\underline{y}))<k_+(\underline{y}) \text{ and } x_0\cdot(1,F'(\overline{y}))>k_-(\overline{y}).
    \end{equation} Since $\nu([0,\overline{y}_\delta])\ge\nu([0,\underline{y}])$ for small enough $\delta,$ and as $f\ge\alpha>0,$ we have $X_\le(\underline{y},k_+(\underline{y}))\subseteq X_<(\overline{y}_\delta,k(\overline{y}_\delta)),$ otherwise $X_=(\overline{y}_\delta,k(\overline{y}_\delta))\subset X_\le(\underline{y},k_+(\underline{y}))$ and since $X_=(\overline{y}_\delta,k(\overline{y}_\delta))$ does not intersect $X_=(\underline{y},k_+(\underline{y}))$ and both have negative slopes, we get
    \[\begin{array}{ll}\nu([0,\overline{y}_\delta])=\mu(X_\le(\overline{y}_\delta,k(\overline{y}_\delta)))&=\mu(X_\le(\underline{y},k_+(\underline{y})))-\mu(X_\le(\underline{y},k_+(\underline{y}))\setminus X_\le(\overline{y}_\delta,k(\overline{y}_\delta)) )\\
    &<\mu(X_\le(\underline{y},k_+(\underline{y})))=\nu([0,\underline{y}])
    \end{array}\] which is a contradiction. From the inclusion $X_\le(\underline{y},k_+(\underline{y}))\subseteq X_<(\overline{y}_\delta,k(\overline{y}_\delta)),$ we get that 
    \begin{equation}\label{deltaineq}x_0\cdot (1,F'(\overline{y}_\delta))< k(\overline{y}_\delta)
    \end{equation} for all $\delta$ small enough. There exists a convergent subsequence $(k(\overline{y}_{\delta_p}))$ such that $k(\overline{y}_{\delta_p})\to \beta$ as $\delta_p \rightarrow 0$ for some $\beta.$ Since $\mu(X_\le(\overline{y}_\delta,k(\overline{y}_\delta)))-\nu([0,\overline{y}_\delta])=0$ for all $\delta,$ as $\delta\to 0,$ we get \[\mu(X_\le(\overline{y},\beta))-\nu([0,\overline{y}))=0\] by the continuity of $\mu(X_\le(y,k))$ and the fact that $\nu([0,\overline{y}_\delta])\to \nu([0,\overline{y})).$ But, $\mu(X_\le(\overline{y},k_-(\overline{y})))=\nu([0,\overline{y})),$ which implies $\beta=k_-(\overline{y}),$ and $k(\overline{y}_{\delta_p})\to k_-(\overline{y})$ as $f\ge\alpha.$ Taking the limit in \eqref{deltaineq} as $p\to\infty,$ we get \[x_0\cdot (1,F'(\overline{y}))\le k_-(\overline{y})\] which  contradicts inequality \eqref{delta} and proves our claim. Let $\underline{y}_\varepsilon=(\underline{y}+\varepsilon,F(\underline{y}+\varepsilon))$ such that $\underline{y}_\varepsilon$ is not an atom. Similar to the previous argument we can prove that for small enough $\varepsilon$ we have $X_=(\underline{y}_\varepsilon,k(\underline{y}_\varepsilon))$ intersects $X_=(\overline{y}_\delta,k(\overline{y}_\delta)).$ Since $\overline{y}_\varepsilon$ and $\overline{y}_\delta$ are not atoms, by  Lemma \ref{noatoms}, $X_=(\overline{y}_\delta,k(\overline{y}_\delta))$ does not intersect $X_=(\underline{y}_\varepsilon,k(\underline{y}_\varepsilon)),$ which is a contradiction. Hence, $X_=(\underline{y},k_+(\underline{y}))$ does not intersect $X_=(\overline{y},k_-(\overline{y})).$ 

As both  $X_=(\underline{y},k_+(\underline{y}))$ and $X_=(\overline{y},k_-(\overline{y}))$ have negative slopes, we have two possibilities, either $X_\le(\underline{y},k_+(\underline{y}))\subset X_<(\overline{y},k_-(\overline{y}))$ or $X_\le(\overline{y},k_-(\overline{y}))\subset X_<(\underline{y},k_+(\underline{y})).$ If $X_\le(\overline{y},k_-(\overline{y}))\subset X_<(\underline{y},k_+(\underline{y})),$ this implies that $\nu([0,\overline{y}))< \nu([0,\underline{y}])$ and that is a contradiction. Then $X_\le(\underline{y},k_+(\underline{y}))\subset X_<(\overline{y},k_-(\overline{y})) $ and therefore the optimal transport problem $(\mu,\nu)$ is nested.  

Turning to the assertion about connectedness of the support, denote by $supp(\nu)$ the support of $\nu.$ Suppose that $supp(\nu)$ is not connected, then there exists $\underline{y},\overline{y}\in supp(\nu)$ such that there exists $\zeta\in[\underline{y},\overline{y}]$ where $\zeta\notin supp(\nu).$ This implies that there exist $\zeta_1,\zeta_2$ such that $\zeta\in(\zeta_1,\zeta_2)\subset [\underline{y},\overline{y}]$ and $\nu((\zeta_1,\zeta_2))=0.$ Then, $\nu([0,\zeta_1])=\nu([0,\zeta_2])\le\nu([0,\overline{y}])\le1.$ Since  $X_{\le}(\zeta_1,k_+(\zeta_1))\subset X_<(\zeta_2,k_-(\zeta_2))$ by the previous part, we get $\mu(X_<(\zeta_2,k_-(\zeta_2))\setminus X_{\le}(\zeta_1,k_+(\zeta_1)))>0$ as the density $f\ge\alpha>0$ and $\nu([0,\zeta_2])<1.$ But, $0=\nu((\zeta_1,\zeta_2))=\mu(X_<(\zeta_2,k_+(\zeta_2))\setminus X_{\le}(\zeta_1,k_-(\zeta_1)))$ which is a contradiction. Therefore, $supp(\nu)$ is connected. 
\end{proof}
We next prove Corollary \ref{c1}.
\begin{proof}
    Since the solution $u$ of the monopolist's problem is nested, and its corresponding $\nu$ has a connected support (due to Proposition \ref{con.nest}), using Theorem 4 in \cite{ChiapporiMcCannPass17}, we conclude that the optimal map $Du$ agrees $\mu-a.e.$ with a continuous function.
\end{proof}

\section{Uniqueness of solutions}\label{uniqueness}
We deal with the alternate formulation of the monopolist's problem introduced in Section 5, and proven to be equivalent to the original in Theorem \ref{p(t)}.

The theorem below provides conditions under which the solution is unique. 
\begin{theorem}\label{thm: uniqueness}
Under the assumption in Theorem \ref{p(t)}, if $\|f_{x_1}\|_\infty\le\alpha$ and $\frac{1}{F'(\Tilde{y})^2} f_{x_2}(1,x_2)+\frac{1}{F'(\Tilde{y})}f_{x_1}(1,x_2)\ge\|f_{x_1x_1}\|_\infty,$  then the optimal $(\overline{t}_i)$ are unique and hence the corresponding $u$ is the unique solution of the monopolist's problem.
\end{theorem}
\begin{proof}
Using Theorem \ref{p(t)}, we can recast the problem as maximizing \[\begin{array}{ll}
 \mathcal{P}(t_0,\dots,t_{M-1})&=\sum_{i=1}^M(v_i-c(z_i))\mu( X_i)\vspace{3pt}\\
 &=\sum_{i=1}^M(\sum_{k=0}^{i-1}(\overline{x}(t_{k})\cdot (z_{k+1}-z_{k}))-c(z_i))\mu( X_i)\vspace{3pt}\\
 &=\sum_{i=1}^M(\sum_{k=0}^{i-1}(t_{k} (y_{k+1}-y_{k})+F(y_{k+1})-F(y_{k}))-c(z_i))\mu( X_i).
 \end{array}\]

 By Lemma \ref{between}, we know that all products between $0$ and $M$ have positive mass  of customers and since the segments $X_=(y_i,k^N(y_i))$ intersect outside $\overline{X}$, then the maximizer is attained at $\overline{t_0}<\overline{t_1}<\dots<\overline{t_M},$ where $(\overline{t_i})$ are critical numbers of $\mathcal{P}.$ After differentiating with respect 
 to $t_i$ we get 
 \begin{equation}\hspace{-2pt}
 \begin{array}{ll}
\frac{\partial\mathcal{P}}{\partial t_i}\vspace{3pt}\\
=&(y_{i+1}-y_i)\sum_{k=i+1}^M\mu( X_k)+\frac{\partial}{\partial t_i}(\mu( X_i))(\sum_{k=0}^{i-1}(x(t_{k})\cdot (z_{k+1}-z_{k}))-c(z_i))\vspace{3pt}\\
&+\frac{\partial}{\partial t_i}(\mu( X_{i+1}))(\sum_{k=0}^{i}(x(t_{k})\cdot (z_{k+1}-z_{k}))-c(z_{i+1}))\vspace{3pt}\\
=&(y_{i+1}-y_i)\sum_{k=i+1}^M\mu( X_k)\vspace{3pt}\\
&+\frac{\partial}{\partial t_i}(\mu( X_i))(\sum_{k=0}^{i-1}(t_{k} (y_{k+1}-y_{k})+F(y_{k+1})-F(y_{k}))-c(z_i))\vspace{3pt}\\
&-\frac{\partial}{\partial t_i}(\mu( X_i))(\sum_{k=0}^{i}(t_{k} (y_{k+1}-y_{k})+F(y_{k+1})-F(y_{k}))-c(z_{i+1}))\vspace{3pt}\\
=&(y_{i+1}-y_i)\sum_{k=i+1}^M\mu( X_k)\vspace{3pt}\\
&+\frac{\partial}{\partial t_i}(\mu( X_i))(c(z_{i+1})-c(z_i)-t_i(y_{i+1}-y_i)-F(y_{i+1})+F(y_i)).\vspace{3pt}\\
=&(y_{i+1}-y_i)\int_{D_i(t_i)}f(x)dx_1\,dx_2\vspace{3pt}\\
&+\cos(\theta_i)\int_{l_i(t_i)}f(x)d\mathcal H^{m-1}(x)(c(z_{i+1})-c(z_i)-t_i(y_{i+1}-y_i)-F(y_{i+1})+F(y_i))\label{tbar}.
\end{array}
\end{equation}
where $D_i(t)=\{x=(x_1,x_2):\, x_2\ge-\frac{y_{i+1}-y_i}{F(y_{i+1})-F(y_i)}(x_1-t)+1\},$  $\theta_i$ is the angle between the $x_1$-axis and the vector $z_{i+1}-z_i$ and $l_i(t)$ is the segment of indifference points between $ X_i$ and $ X_{i+1}$ (the line $x_2=-\frac{y_{i+1}-y_i}{F(y_{i+1})-F(y_i)}(x_1-t)+1$) and we use  $\frac{\partial}{\partial t_i}(\mu( X_i))=\frac{d}{ds}(\mu( X_i))\frac{ds}{dt_i}=\cos(\theta_i)\int_{l_i}f(x)d\mathcal H^{m-1}(x),$ and $s$ is the variable in the direction of $z_{i+1}-z_i$ which is perpendicular to $l_i.$  

 At this point, we note that $\frac{\partial\mathcal{P}}{\partial t_i}$ depends on $t_i$ but not on any $t_j$ for $j \neq i$.  The Hessian of $\frac{\partial\mathcal{P}}{\partial t_i}$ is therefore diagonal, and to assert its strict concavity on a region one must only show that those diagonal elements $\frac{\partial^2\mathcal{P}}{\partial t_i^2}$ are negative.

We consider two cases. The first case is when $l_i$ reaches the $x_1-$axis, then $t_i$ lies in $\Big[0,1-\frac{F(y_{i+1})-F(y_i)}{y_{i+1}-y_i}\Big]$ as $\Big(1-\frac{F(y_{i+1})-F(y_i)}{y_{i+1}-y_i},1\Big)$ is the intersection between $[0,1]\times\{1\}$ and the line passing through $(1,0)$ with slope equals to $-\frac{y_{i+1}-y_i}{F(y_{i+1})-F(y_i)}.$ We then have
\[\begin{array}{ll}
\frac{\partial\mathcal{P}}{\partial t_i}&=(y_{i+1}-y_i)\int_{D_i(t_i)}f(x)dx_1\,dx_2\vspace{3pt}\\
&\quad +(c(z_{i+1})-c(z_i)-t_i(y_{i+1}-y_i)-F(y_{i+1})+F(y_i))\times \vspace{3pt}\\
&\quad\quad\int_0^1 f\Big(t_i+(1-x_2)\frac{F(y_{i+1})-F(y_{i})}{y_{i+1}-y_{i}},x_2\Big)dx_2
\end{array}\]

Note that any critical point in this case must lie in the region where $ t_i>\underline{t}_i:=\frac{c(z_{i+1})-c(z_i)-(F(y_{i+1})-F(y_i))}{y_{i+1}-y_i}$.
We differentiate the expression with respect to $t_i$ to get

\[\begin{array}{ll}
\frac{\partial^2\mathcal{P}}{\partial t_i^2}&=-2(y_{i+1}-y_i)\int_0^1 f\Big(t_i+(1-x_2)\frac{F(y_{i+1})-F(y_{i})}{y_{i+1}-y_{i}},x_2\Big)dx_2\vspace{3pt}\\
&\quad-(t_i(y_{i+1}-y_i)+F(y_{i+1})-F(y_i)-(c(z_{i+1}-c(z_i)))\times\vspace{3pt}\\
&\quad\quad\int_0^1 f_{x_1}\Big(t_i+(1-x_2)\frac{F(y_{i+1})-F(y_{i})}{y_{i+1}-y_{i}},x_2\Big)dx_2\vspace{3pt}\\
&\le(-2(y_{i+1}-y_i)+t_i(y_{i+1}-y_i)+F(y_{i+1})-F(y_i)-(c(z_{i+1})-c(z_i))))\times \vspace{3pt}\\
&\quad\int_0^1 f\Big(t_i+(1-x_2)\frac{F(y_{i+1})-F(y_{i})}{y_{i+1}-y_{i}},x_2\Big)dx_2<0,\vspace{3pt}\\

\end{array}\]
for all $t_i>\underline{t}_i,$ where 
we used the fact that

$$\int_0^1 \Big|f_{x_1}\Big(t_i+(1-x_2)\frac{F(y_{i+1})-F(y_{i})}{y_{i+1}-y_{i}},x_2\Big)\Big|dx_2 \le \alpha \le \int_0^1 f\Big(t_i+(1-x_2)\frac{F(y_{i+1})-F(y_{i})}{y_{i+1}-y_{i}},x_2\Big)dx_2,$$

$t_i < 1$ and $F(y_{i+1})-F(y_i)<c(z_{i+1})-c(z_i)$ due to condition \ref{condition2} and our assumption $c(z_1) >F(z_1)$.

Hence, $\frac{\partial^2\mathcal{P}}{\partial t_i^2}<0$ for all 
$t_i\in\Big[\underline{t}_i,1-\frac{F(y_{i+1})-F(y_i)}{y_{i+1}-y_i}\Big]$. 

When $l_i$ reaches $\{1\}\times[0,1],$ $t_i$ lies in $\Big[1-\frac{F(y_{i+1})-F(y_i)}{y_{i+1}-y_i},1\Big].$ Then,
\[\begin{array}{ll}\frac{\partial\mathcal{P}}{\partial t_i}&=(y_{i+1}-y_i)\int_{D_i(t_i)}f(x)dx_1\,dx_2\vspace{3pt}\\
&\quad+(c(z_{i+1})-c(z_i)-t_i(y_{i+1}-y_i)-F(y_{i+1})+F(y_i))\times \vspace{3pt}\\
&\quad\quad\int_{r_i}^1 f\Big(t_i+(1-x_2)\frac{F(y_{i+1})-F(y_{i})}{y_{i+1}-y_{i}},x_2\Big)dx_2\
\end{array}\]
where $r_i=1-\frac{y_{i+1}-y_i}{F(y_{i+1})-F(y_i)}(1-t_i).$ We differentiate to get
\[\begin{array}{ll}
\frac{\partial^2\mathcal{P}}{\partial t_i^2}&=-2(y_{i+1}-y_i)\int_{r_i}^1 f\Big(t_i+(1-x_2)\frac{F(y_{i+1})-F(y_{i})}{y_{i+1}-y_{i}},x_2\Big)dx_2\vspace{3pt}\\
&\quad-(t_i(y_{i+1}-y_i)+F(y_{i+1})-F(y_i)-(c(z_{i+1}-c(z_i)))\times\vspace{3pt}\\
&\quad\Big(-\frac{y_{i+1}-y_i}{F(y_{i+1})-F(y_i)}f(1,r_i)+\int_{r_i}^1 f_{x_1}\Big(t_i+(1-x_2)\frac{F(y_{i+1})-F(y_{i})}{y_{i+1}-y_{i}},x_2\Big)dx_2\Big)\vspace{3pt}\\
\end{array}\]
This expression may not always be negative; however, we will show that it is increasing in $t_i$.  Differentiating again, and using the rage of $t_i,$ we get

\[\begin{array}{ll}
    
\frac{\partial^3\mathcal{P}}{\partial t_i^3}&=3(y_{i+1}-y_i)\Big(\frac{y_{i+1}-y_i}{F(y_{i+1})-F(y_i)}f(1,r_i)-\int_{r_i}^1 f_{x_1}\Big(t_i+(1-x_2)\frac{F(y_{i+1})-F(y_{i})}{y_{i+1}-y_{i}},x_2\Big)dx_2\Big)\vspace{3pt}\\
&\quad+(t_i(y_{i+1}-y_i)+F(y_{i+1})-F(y_i)-(c(z_{i+1}-c(z_i)))\times\vspace{3pt}\\
&\quad\quad\Bigg(\Big(\frac{y_{i+1}-y_i}{F(y_{i+1})-F(y_i)}\Big)^2f_{x_2}(1,r_i)+\frac{y_{i+1}-y_i}{F(y_{i+1})-F(y_i)}f_{x_1}(1,r_i)\vspace{3pt}\\
&\quad-\int_{r_i}^1 f_{x_1x_1}\Big(t_i+(1-x_2)\frac{F(y_{i+1})-F(y_{i})}{y_{i+1}-y_{i}},x_2\Big)dx_2\Bigg)>0\vspace{3pt}\\

\end{array}\]

using the assumptions on $f$ where
\[\begin{array}{ll}
    
&\Big(\frac{y_{i+1}-y_i}{F(y_{i+1})-F(y_i)}\Big)^2f_{x_2}(1,r_i)+\frac{y_{i+1}-y_i}{F(y_{i+1})-F(y_i)}f_{x_1}(1,r_i)\vspace{3pt}\\
&\quad-\int_{r_i}^1 f_{x_1x_1}\Big(t_i+(1-x_2)\frac{F(y_{i+1})-F(y_{i})}{y_{i+1}-y_{i}},x_2\Big)dx_2\vspace{3pt}\\
&>\frac{1}{F'(\Tilde{y})^2} f_{x_2}(1,r_i)+\frac{1}{F'(\Tilde{y})}f_{x_1}(1,r_i)-\int_{r_i}^1 f_{x_1x_1}\Big(t_i+(1-x_2)\frac{F(y_{i+1})-F(y_{i})}{y_{i+1}-y_{i}},x_2\Big)dx_2.\vspace{3pt}\\

\end{array}\]

Therefore, the function $t_i \mapsto \mathcal{P}(t_0,\dots,t_{M-1})$ can have at most one inflection point.  Furthermore, the inflection point, if it exists, does not depend on the other $t_j$, since, as noted above $\frac{\partial\mathcal{P}}{\partial t_i}$ and hence $\frac{\partial^2\mathcal{P}}{\partial t_i^2}$ does not depend on $t_j$ for $j \neq i$.

We define $\hat t_i$ to be the inflection point, that is, $\frac{\partial^2\mathcal{P}}{\partial t_i^2}(t_0,\dots \hat t_i,\dots,t_{M-1})=0$ if it exists, $\hat t_i =1$ otherwise.   Note that $\hat{t}_i\in\Big[1-\frac{F(y_{i+1})-F(y_i)}{y_{i+1}-y_i},1\Big].$
From the definition of nestedness and Lemmas \ref{case1},\ref{case2},\ref{case3} and \ref{m}, we get that any maximizer $\overline t= (\overline{t}_1,...,\overline t_{M-1})$ of $\mathcal{P}$ is a local maximizer, and therefore $\frac{\partial^2\mathcal{P}}{\partial t_i^2}(\overline{t}_1,...,\overline t_{M-1}) \leq 0$ for each $i$.  Therefore, every maximizer lies in $A=\cap_{i=0}^{M-1}A_i$, where $A_i=\{(t_0,\dots,t_{M-1}),\, \underline{t}_i<t_i\leq \hat{t}_i\}\cap\mathcal{B}$. However,  $\mathcal{P}$ is strictly concave on the convex set $A$, which implies uniqueness of the maximizer.
\end{proof}

\bibliographystyle{plain}  
\bibliography{references}

\end{document}